\documentclass[11pt]{amsart}

\usepackage{amsfonts}
\usepackage{graphicx}
\usepackage{amsmath}
\usepackage{amssymb}

 \advance\hoffset by -0.0 truecm

\newtheorem{theorem}{Theorem}

\newtheorem{lemma}{Lemma}
\newtheorem{proposition}{Proposition}
\theoremstyle{definition}

\theoremstyle{remark}

\theoremstyle{note}

\numberwithin{equation}{section}

\newcommand{\im}{\hbox{Im}}

\DeclareMathOperator{\spn}{span} \DeclareMathOperator{\Cl}{Cl}
 
 \DeclareMathOperator{\End}{End}

\begin{document}
\title[Sub-semi-Riemannian geometry on $H$-type groups]
{Sub-semi-Riemannian geometry on $H$-type groups}

\author{Anna Korolko}

\address{Department of Mathematics,
University of Bergen, Johannes Brunsgate 12, Bergen 5008, Norway}

\email{anna.korolko@uib.no}

\thanks{The author is partially supported by the grant of the Norwegian Research Council \# 177355/V30, by the grant of the European Science Foundation Networking Programme HCAA, and by the NordForsk Research Network Programme \# 080151}

\subjclass[2000]{53C50,\ 53B30\, 53C17}

\keywords{$H$-type group, sub-semi-Riemannian geometry, geodesic, Clifford algebra}


\begin{abstract}

We consider $H$(eisenberg)-type groups whose law of left translation gives rise to a bracket generating distribution of 
step 2. In the contrast with 
sub-Riemannian studies we furnish the horizontal distribution with a nondegenerate indefinite metric of arbitrary index 
and investigate the problem concerning causal geodesics on underlying manifolds. The exact formulae for geodesics are obtained.
\end{abstract}

\maketitle

\section{Introduction}\label{sec:1}

Sub-semi-Riemannian geometry is a study of nonholonomic manifolds $M$ whose non-integrable tangent subbundles $\mathcal{D}$ of the 
tangent bundle $TM$ are equipped with a nondegenerate indefinite metric $g$ (sub-semi-Riemannian metric). 
If the metric $g$ on $\mathcal{D}$ is a positively 
definite metric, then the triple $(M,\mathcal{D},g)$ is known as sub-Riemannian manifold which has become an actively studied 
mathematical topic. If $g$ is a nondegenerate 
indefinite metric on $\mathcal{D}$ of index $p\in \mathbb{N}$, then the corresponding object can be called a sub-semi-Riemannian 
manifold; if $p=1$, the manifold is called sub-Lorentzian. If $\mathcal{D}$ coincides with the entire tangent bundle $TM$ and $g$ is nondegenerate 
metric of index greater than 1, then the 
corresponding triplet reduces to semi-Riemannian manifold. 
There are numerous applications of nonholonomic manifolds with positively definite metric in problems in physics of high energy, 
control theory, electromagnetism, mechanical and biological systems with nonholonomic constraints. Therefore, it is interesting to start 
studying  more generalized objects -- sub-semi-Riemannian manifolds.

Most of notions in sub-semi-Riemannian geometry originate in sub-Riemannian geometry \cite{Montgomery,Chang} and Lorentzian 
geometry \cite{Oneill}.

The intention of the article is to enlighten some aspects of sub-semi-Riemannian geometry. We start with a 
study of geometry, analysis and underlying algebraic structures of $H$(eisenberg)-type groups with 
nondegenerate metric. These 2-step nilpotent Lie groups are interesting and natural examples of nonholonomic manifolds 
in sub-semi-Riemannian sense and at the same time first large class of objects with similar behavior which can be described in 
sub-semi-Riemannian geometry. The notion of $H$-type groups with 
positively definite metric was first introduced 
by A.~Kaplan in his works \cite{Kap,Kap1,Kap2} with the help of theory of Clifford algebras. 
Such groups with a Riemannian metric provide nice 
examples of nonholonomic manifolds to study for a sub-Riemannian geometry \cite{CM,CM2}.

Some examples of $H$-type groups with semi-Riemannian metric on the distribution were already considered in \cite{Groch,KM,KM2,KM3} 
and this is the first attempt to obtain a generalizing result concerning such geometric objects of interest as geodesics. 
We solve the geodesic equation for sub-semi-Riemannian geodesics, which coincides  with the geodesic equation in 
sub-Riemannian case if the 
considered metric tensor is Riemannian. 
The essential role here plays the skew-symmetry of emerging objects with respect to the 
semi-Riemannian product on $V$, which is not positively definite anymore (see the section \ref{Geodesics}). 
The skew-symmetry arises from the use of Clifford algebras and their representations. 
A good deal of the article is devoted to the description of skew-symmetric transformations w.~r.~t. semi-Riemannian metric. 
The canonical view and, therefore, eigenvalues of corresponding matrices that are skew-symmetric w.~r.~t. semi-Riemannian metric $g$ 
greatly depends on the index of $g$.

The article is organized in the following way. In Section \ref{General} we introduce $H$-type groups endowed with semi-Riemannian 
metric. We also present necessary basic notions from space-time geometry. In Section \ref{Exp} we indroduce exponential map, with 
its help we calculate left-invariant vector fields and the connection. Section \ref{Geodesics} is 
devoted to calculations of geodesics and auxiliary algebraic backgrounds. And in Section \ref{Example} we provide the example of
 $H$-type group with 
semi-Riemannian metric which illustrates how the expounded theory could be applied for the problem of finding geodesics in a 
particular case. We conclude the survey by the list of open problems in Section 6.

\section{General structure of Heisenberg-type groups with nondegenerate metric}\label{General}

Let $U$ and $V$ be vector spaces of dimensions $m$ and $n$ respectively with positively definite 
product $\langle\cdot,\cdot\rangle_U$ 
on $U$ and nondegenerate product  $\langle\cdot,\cdot\rangle_V$ of signature $(p,q)$, $p+q=n$, on $V$.
Denote by $\eta$ the metric tensor of the product on $V$, which is a diagonal matrix with first $p$ negative unities and 
then $q$ positive unities, i.~e.  
\begin{equation}\label{eq:eta}
    \eta=
\left(\begin{array}{cc}
   -I_p&0\\
   0&I_q
\end{array}\right),
\end{equation}
where $I_p$ and $I_q$ are usual identity matrices of dimensions $(p\times p)$ and $(q\times q)$ respectively.
The corresponding metric $g\colon V\times V\to\mathbb{R}$ is called a nondegenerate metric of index $p$, $p\in\mathbb{N}$. 
Then as usually $g(v,w)=\langle v,w\rangle_V=w^t\eta v$, where $v,w\in V$ and $w^t$ denotes the vector transposed to $w$. 
Nondegeneracy is understood here in the sense that if $\langle v,w\rangle_V=0$ for any $w\in V$, then $v=0$.
If the index of the metric is 1, then it is called the Lorentzian metric. The metric with the index 0 is just a Riemannian metric. 
In the case of arbitrary index such a nondegenerate indefinite metric  received the name semi-Riemannian metric.

Let us denote by the symbol $\End(V)$ the set of linear maps from $V$ to $V$.
Consider a linear mapping $j\colon U\to \End(V)$, which satisfies the following properties:
\begin{equation}\label{eq:j}
\begin{split}
   &i) \quad j^2(u)=-|u|^2_U \,I_V,\\
   &ii) \quad \langle \eta j(u)v,w\rangle_V= -\langle v,\eta j(u)w\rangle_V,
\end{split}
\end{equation} where $u\in U$, $v,w\in V$ and $I_V$ is a $(n\times n)$ identity matrix in $V$, $|u|^2_U=\langle u,u\rangle_U$. 
Conditions \eqref{eq:j} are taken from \cite{Kap2} and adapted for the case of semi-Riemannian metric on $V$.
 For all $u\in U$ the linear map $j(u)$ 
provides a complex structure on $V$ 
compatible with its euclidean structure.
By abuse of notations we use the symbol $j$ to denote both the mapping and the corresponding $(n\times n)$-matrix.

{\bf Remark}. 
The property of $j$ being skew-symmetric remains the same as in the 
case of $H$-type groups with positively definite metrics on both spaces $U$ and $V$ (\cite{Kap2}).
Indeed, notice that 
\begin{gather*}
 \langle \eta j(u)v,w\rangle_V=w^t\eta\eta j(u)v=w^tj(u)v,\\
 \langle v,\eta j(u)w\rangle_V=w^t j(u)^t\eta^t\eta v=w^tj(u)^tv
\end{gather*} and from the condition $ii)$ it follows that
\begin{equation}\label{eq:j-skew}
   j(u)=-j(u)^t.
\end{equation}

Then we can define the bilinar skew-symmetric map $[\cdot,\cdot]\colon V\times V\to U$ by 
\begin{equation}\label{eq:bracket}
    \langle u,[v,w]\rangle_U =\langle\eta j(u)v,w\rangle_V.
\end{equation}
We can now define a Lie algebra $$\mathcal{G}=U\oplus V$$ with a Lie structure on it:
$$\qquad\qquad[u+v,r+w]=[v,w]\qquad (u,r\in U, \; v,w\in V).$$
Constructed in this way Lie algebra is graded and 2-step nilpotent with the center $U$. As in \cite{Kap2} we shall refer to 
these algebras as $H$-type algebras with nondegenerate metric $\langle\cdot,\cdot\rangle_V$.

Recall that if $U$ is a $m$-dimensional real space $\mathbb{R}^m$ with a basis consisting of $m$ orthogonal unit vectors $e_1,
\ldots, e_m$, then the Clifford algebra $\Cl(U,m)$ is the real associative algebra generated by the set $\{e_1,\ldots,e_m\}$, 
satisfying 
the relations (see \cite{Lounesto})
\begin{gather*}
   e_{\alpha}^2=-1,\quad \alpha=1,\ldots,m,\\
   e_{\alpha}e_{\beta}=-e_{\beta}e_{\alpha},\quad\mbox{for}\quad \alpha\neq \beta.
\end{gather*} 
The Clifford algebra $\Cl(U,m)$ is $2^m$-dimensional with the following basis:
\begin{gather*}
    \begin{array}{cc}
      1&\text{the scalar}\\
      e_1,\ldots, e_m &\text{vectors}\\
      e_1e_2,\ldots,e_{m-1}e_m&\text{bivectors}\\
      e_1e_2e_3,\ldots,e_{m-2}e_{m-1}e_m&$3-$\text{vectors}\\
      \vdots&\vdots\\
      e_1e_2\ldots e_m &\text{the } $m-$\text{volume element}.
    \end{array}
\end{gather*} Hence an arbitrary element of the basis may be written as $e_{\mathcal{A}}=e_{\alpha_1}\ldots e_{\alpha_i}$; 
here $\mathcal{A}=(\alpha_1,\ldots,\alpha_i)$ and $1\leqslant \alpha_1<\alpha_2<\ldots<\alpha_i\leqslant m$.
All elements of this type
share the basic property of the imaginary unit of being the square root of $-1$, i.~e. $(e_{\mathcal{A}})^2=-1$. It is known that real  
Clifford algebras $\Cl(U,m)$ are isomorphic to real matrix algebras with entries in $\mathbb{R}$, $\mathbb{C}$, $\mathbb{H}$ or in 
$\mathbb{R}\oplus\mathbb{R}$, $\mathbb{H}\oplus\mathbb{H}$, where by $\mathbb{R}$, $\mathbb{C}$ and $\mathbb{H}$ we denote 
division rings of real numbers, complex numbers and of quaternions respectively. 

The mapping $j\colon U\to \End(V)$ satisfying \eqref{eq:j} extends to a representation of Clifford algebra $\Cl(U,m)$ on $V$; in 
other words, the space $V$ becomes a Clifford module over $\Cl(U,m)$. The images of the generators $e_1,\ldots,e_m$ of a Clifford algebra 
$\Cl(U,m)$ under a representation $j$: $j_{\alpha}:=j(e_{\alpha})$, $\alpha=1,\ldots,m$, are called {\it(standard) generators} of the 
representation. It was noticed that without loss of generality one can represent standard generators as a set of anti-commuting 
skew-symmetric matrices with squares $-I_V$ (see \cite{Bilge}). Moreover, these matrices can be expressed as homogeneous tensor 
products of the standard Pauli ($2\times2$)-matrices, though the determination of such matrices is nontrivial.

Recall also that $V$ admits a structure of Clifford module over 
$\Cl(U,m)$ if and only if the following 
condition on dimensions $n$ and $m$ of vector spaces $V$ and $U$, respectively, holds: $0\leqslant m<\rho(n)$, where $\rho$ is the 
Hurwitz-Radon function defined by $n=k2^{4r+s}\mapsto\rho(n)=8r+2^s$, $k$ -- odd, $0\leqslant s\leqslant 3$ \cite{Kap2}. 
If $V$ is a normed division algebra, then $U=\im V$ is a space of imaginary elements of $V$ and $\dim U=m=n-1$.
If $m=0$ and $n$ is arbitrary, then $\mathcal{G}$ is the Euclidean $n$-dimensional space; if $m=1$ and $n$ is even, 
then $\mathcal{G}$ is the Heisenberg algebra; if $m=3$ and $n=4k$, then 
$\mathcal{G}$ is the Quaternion algebra; if $m=7$ and $n=8$, then we call the corresponding algebra an 
$\mathbb{H}\oplus\mathbb{H}$-algebra. These four algebras generate a special class among the $H$-type algebras, satisfying the $j^2$-condition: 
for all $u_1,u_2\in U$ with $\langle u_1,u_2\rangle_U=0$ and all nonzero $v\in V$ there exists a $u_3\in U$ so 
that 
\begin{equation*}
  j(u_1)j(u_2)v=j(u_3)v.
\end{equation*}

Exact formulae for geodesics on some of corresponding groups with semi-Riemannian metric on $V$ were found in \cite{KM,KM2,KM3}. 
In the present 
article we concentrate our attention on the general $H$-type groups with nondegenerate metric on $V$.

Let us consider the $H$-type group $G$ with the Lie algebra $\mathcal{G}=V\oplus U$ and nondegenerate product $\langle\cdot,\cdot
\rangle_V$ on $V$. The Lie algebra $\mathcal{G}$ is identified with the tangent space to $G$ at the identity $T_eG$. 
The scalar product $\langle\cdot,\cdot\rangle$ on the algebra $\mathcal{G}$ is a sum of the inner product $\langle\cdot,\cdot\rangle_U$ on $U$ 
and nondegenerate product $\langle \cdot,\cdot\rangle_V$ on $V$. 
Push-forward allows to define the 
scalar product on the whole tangent bundle $TG=\cup_{\sigma}T_{\sigma}G$, not only in $T_eG$:
$$\langle\cdot,\cdot\rangle(\sigma)=\langle dL_{\sigma^{-1}}\cdot,dL_{\sigma^{-1}}\cdot\rangle (e),$$
 where $L_{\sigma}$ denotes the left translation on $G$ 
by the element $\sigma\in G$. 
Since $V\subset T_eG$, the left translation of $V$ by element $\sigma$ is $dL_{\sigma}(V)\subset T_{\sigma}G$. The mapping 
$\mathcal{D}\colon \sigma \mapsto dL_{\sigma}(V)$ is called a {\it horizontal distribution} and any vector $v(\sigma)\in dL_{\sigma}
(V)$ is called a {\it horizontal vector} at point $\sigma$.  

By analogy with sub-Riemannian geometry the absolutely continuous curve $c(s)\colon [0,1]\to \mathcal{D}$ 
is called the {\it horizontal curve} if the 
tangent vector $\dot c(s)$ belongs to $\mathcal{D}$ at each $s\in[0,1]$. The Lie group $G$ was introduced in such a way, that 
the horizontal distribution $\mathcal{D}$ possesses a {\it 2-step bracket 
generating property}. It means that for any point $p\in G$ we have that any tangent vector can be presented as a linear combination 
of the vectors of the following types $v(p),\,\,[v,w](p)\in T_pM$, where all vectors $v,\,w$ are horizontal. 
A result of Chow \cite{Chow} says that any two points on $G$ can be connected by 
a piecewise smooth horizontal curve because of the bracket generating property of $\mathcal{D}$. 
Chow theorem guarantees that the connectivity of a manifold by curves tangent 
to a given distribution $\mathcal{D}$ does not depend neither on metric defined on it nor on tangent bundle.

In this article we study the triple $(G,\mathcal{D},g)$, which we call the {\it sub-semi-Riemannian manifold},  where 
the nondegenerate bundle-type metric $g$ 
of arbitrary index $p$ is given on distribution $\mathcal{D}$. We will call the corresponding group $G$ sub-semi Riemannian $H$-type group. 
Due to left translation we can work further 
with the distribution at the point $e$, i.~e. with   
$dL_e(V)=V$ instead of $\mathcal{D}$. 

The nondegeneracy of the metric $g$ brings the causal character of the vectors in $V$. 
The notations are carried from the Lorentzian geometry \cite{Beem}. 
Even though the metric $g$ has the arbitrary index $p$ and the time 
orientation can not be defined in general (unless $p=1$), we can anyway introduce the notions of timelikeness, spacelikeness and 
lightlikeness.
A horizontal vector $v\in V$ is called {\it timelike} if $\langle v,v\rangle_V<0$, {\it spacelike} if 
$\langle v,v\rangle_V>0$ or $v=0$, {\it lightlike} if $\langle v,v\rangle_V=0$ and $v\neq0$, 
{\it nonspacelike} if $\langle v,v\rangle_V\leqslant0$. A horizontal curve is called timelike if its tangent vector
is timelike at each point; similarly for spacelike, lightlike and nonspacelike.


\section{Exponential map, left-invariant vector fields and connection}\label{Exp}

We define the exponential map as a function $\exp: \mathcal{G}\to G$ in a usual way in the Lie groups theory. 
Then we can define smooth maps $\bf{u}$$\colon G\to U$ and $\bf{v}$$\colon G\to V$ such that
\begin{equation*}
   \sigma=\exp({\bf u}(\sigma)+{\bf v}(\sigma))\quad\mbox{for}\quad \sigma\in G.
\end{equation*}

Since $\mathcal{G}$ is a 2-step nilpotent Lie algebra, then the Baker-Campbell-Hausdorff formula yields
\begin{equation}\label{eq:BCHformula}
     \begin{split}
        &{\bf{v}}(\sigma_1\sigma_2)={\bf{v}}(\sigma_1)+{\bf{v}}(\sigma_2),\\        
        &{\bf{u}}(\sigma_1\sigma_2)={\bf{u}}(\sigma_1)+{\bf{u}}(\sigma_2)+\frac{1}{2}[{\bf{v}}(\sigma_1),{\bf{v}}(\sigma_2)].
     \end{split}
\end{equation}
Suppose that $\left\{\partial_{v_j}\right\}_{j=1}^n$ and 
$\left\{\partial_{u_{\alpha}}\right\}_{\alpha=1}^m$ are orthonormal bases for $V$ and $U$
 respectively such that $\langle \partial_{v_i}, 
\partial_{v_j}\rangle_V=\varepsilon_j\delta_{ij}$ and 
$\langle \partial_{u_{\alpha}}, 
\partial_{u_{\beta}}\rangle_U=\delta_{\alpha\beta}$, $\langle \partial_{v_i},\partial_{u_{\alpha}}\rangle=0$, $\delta_{ij}$ is a usual Kronecker 
symbol and $\varepsilon_j=-1$ for $1\leqslant j\leqslant p$ and $\varepsilon_j=1$ for $p+1\leqslant j\leqslant p+q=n$.
Denote by $A^{\alpha}_{ij}$ the analog of Clifford coefficients, i.~e. the coefficients in the decomposition 
$$\eta j\big(\partial_{u_{\alpha}}\big)\partial_{v_i}=\sum\limits_{j=1}^nA^{\alpha}_{ij}\partial_{v_j}\in V.$$ 
Denote by $B^{\beta}_{ij}$ the structure constants in 
$$\Big[\partial_{v_i},\partial_{v_j}\Big]=\sum\limits_{\beta=1}^mB^{\beta}_{ij}
\partial_{u_{\beta}}\in U.$$
They remain the same as in case of $H$-type groups with sub-Riemannian metric.
The structure constants $B^{\beta}_{ij}$ are connected to the analog of Clifford coefficients in a following way:
\begin{lemma}
   $$B^{\alpha}_{ij}=\varepsilon_jA^{\alpha}_{ij},$$
\end{lemma}
\begin{proof}
We calculate
 \begin{equation}\label{eq:B}
   \langle\eta j\big(\partial_{u_{\alpha}}\big)\partial_{v_i},
 \partial_{v_j}\rangle_V=\langle \partial_{u_{\alpha}},\big[
 \partial_{v_i},\partial_{v_j}\big]\rangle_U=
   \sum\limits_{\beta=1}^mB^{\beta}_{ij}\langle \partial_{u_{\alpha}}, \partial_{u_{\beta}}\rangle_U=
   \sum\limits_{\beta=1}^mB^{\beta}_{ij}\delta_{\alpha\beta}=B^{\alpha}_{ij},
\end{equation}
but from the other side
\begin{gather*}
     \langle\eta j\big(\partial_{u_{\alpha}}\big)\partial_{v_i},\partial_{v_j}\rangle_V=
    \langle\sum\limits_{k=1}^nA^{\alpha}_{ik}\partial_{v_k},\partial_{v_j}\rangle_V
    =\sum\limits_{k=1}^nA^{\alpha}_{ik}\langle \partial_{v_k},\partial_{v_j}\rangle_V
    =\sum\limits_{k=1}^nA^{\alpha}_{ik}\varepsilon_k\delta_{kj}=\varepsilon_jA^{\alpha}_{ij}.
\end{gather*}
\end{proof}

Let us denote the global coordinates on the group $G$ by $\{v_1,\ldots,v_n,u_1,\ldots, u_m\}$: 
$${\bf{v}}(\sigma)=\sum\limits_{i=1}^nv_i\partial_{v_i}(e),\qquad 
{\bf{u}}(\sigma)=\sum\limits_{\alpha=1}^mu_{\alpha}\partial_{u_{\alpha}}(e).$$
Let us compute the Lie bracket
\begin{gather*}
    [{\bf{v}}(\sigma_1),{\bf{v}}(\sigma_2)]=\Big[\sum\limits_{i=1}^nv_i^1\partial_{v_i}(e),
    \sum\limits_{j=1}^nv_j^2\partial_{v_j}(e)\Big]\notag\\
    = v_1^2\Big(\sum\limits_{j=1}^nv_j^1B^1_{j1}\partial_{u_1}
    +\sum\limits_{j=1}^nv_j^1B^2_{j1}\partial_{u_2}+\ldots
    +\sum\limits_{j=1}^nv^1_jB^m_{j1}\partial_{u_m}\Big)\notag\\
     +v_2^2\Big(\sum\limits_{j=1}^nv_j^1B^1_{j2}\partial_{u_1}
    +\sum\limits_{j=1}^nv_j^1B^2_{j2}\partial_{u_2}+\ldots
    +\sum\limits_{j=1}^nv^1_jB^m_{j2}\partial_{u_m}\Big)\notag\\
    \vdots\notag\\
    +v_n^2\Big(\sum\limits_{j=1}^nv_j^1B^1_{jn}\partial_{u_1}
   +\sum\limits_{j=1}^nv_j^1B^2_{jn}\partial_{u_2}+\ldots
   +\sum\limits_{j=1}^nv^1_jB^m_{jn}\partial_{u_m}\Big).
\end{gather*}
Therefore, the left-invariant vector fields obtained by the rule $V_i(\sigma)=dL_{\sigma}
 \partial_{v_i}(e)$ are as follows
\begin{equation*}
      V_i=\partial_{v_i}
    +\frac{1}{2}\sum\limits_{\alpha=1}^m\Big(\sum\limits_{j=1}^nv_j
    B^{\alpha}_{ji}\Big)\partial_{u_{\alpha}}, \quad i=1,\ldots,n
\end{equation*}
and $$U_{\alpha}=\partial_{u_{\alpha}}, \quad \alpha=1,\ldots m.$$
The vector fields remain the same as in sub-Riemannian case when the metric on $V$ is positively definite, 
since the matrix $dL_{\sigma}$ is obtained from the group structure, which we do not change (compare with \cite{Kap2}).
Matrices $B^{\alpha}=\big(B^{\alpha}_{ji}\big)$, 
$\alpha=1,\ldots,m$, $i,j=1,\ldots,n$,
are in one-to-one correspondance defined by \eqref{eq:B} with matrices $j_{\alpha}=j(\partial_{u_{\alpha}})$ 
that are the standard generators of representation 
for Clifford algebra $\Cl(U,m)$.
Analogously to sub-Riemannian case 
we can write the horizontal gradient, i.~e. the vector $\mathcal{V}=
(V_1,\ldots,V_n)$, in the form
$$\mathcal{V}=\triangledown_v+\frac{1}{2}\Big(B^1v\partial_{u_1}+\ldots+B^mv\partial_{u_m}\Big),$$
with $v=(v_1,\ldots,v_n)$ and $\triangledown_v=\big(\partial_{v_1},\ldots,\partial_{v_n}\big)$. 

{\bf Remark.}
Notice that $\{V_i, U_{\alpha}\}$, $i=1,\ldots,n,\,\alpha=1,\ldots,m$ is an orthonormal basis with respect to the left-invariant 
metric $\langle\cdot,\cdot\rangle$.
\begin{proof}
By definition of metric $\langle\cdot,\cdot\rangle$ on algebra $\mathcal{G}$ we have:
 \begin{gather*}
     \langle V_i,V_j\rangle=\langle V_i,V_j\rangle_V=\langle \partial_{v_i},\partial_{v_j}\rangle_V=\varepsilon_j\delta_{ij},\\
     \langle U_{\alpha},U_{\beta}\rangle=\langle U_{\alpha},U_{\beta}\rangle_U=
     \langle \partial_{u_{\alpha}},\partial_{u_{\beta}}\rangle_U=\delta_{\alpha\beta},\\
     \langle V_i,U_{\alpha}\rangle=0.
\end{gather*}
\end{proof}

The connection on a group $G$ is defined by Koszul formula (\cite{Oneill})
\begin{gather*}
      \langle \bigtriangledown_XY,Z\rangle=\frac{1}{2}\Big(
    X\langle Y,Z\rangle+Y\langle Z,X\rangle-Z\langle X,Y\rangle
-\langle X,[Y,Z]\rangle+\langle Y,[Z,X]\rangle+\langle Z,[X,Y]\rangle\Big).
\end{gather*}
For the orthonormal basis of $\mathcal{G}$ this formula receives the form
\begin{equation*}
      \langle \bigtriangledown_XY,Z\rangle=\frac{1}{2}\Big(-\langle X,[Y,Z]\rangle+\langle Y,[Z,X]\rangle+\langle Z,[X,Y]\rangle\Big),
\end{equation*}
where $X,Y,Z$ are arbitrary basis elements in the corresponding Lie algebra $\mathcal{G}=U\oplus V$. 
In the case of $H$-type group we calculate that
\begin{equation*}\begin{split}
    &\langle \bigtriangledown_{V_i}V_j,U_{\alpha}
    \rangle=\frac{1}{2}\langle U_{\alpha},[V_i,V_j]
   \rangle_U=\frac{1}{2}\langle \eta j\big(U_{\alpha}\big)V_i,
   V_j\rangle_V,\\
    &\langle \bigtriangledown_{U_{\alpha}}V_i,
   V_j\rangle=-\frac{1}{2}\langle U_{\alpha},
   \big[V_i,V_j\big]\rangle_U
    =-\frac{1}{2}\langle \eta j\big(U_{\alpha}\big)V_i,
   V_j\rangle_V,\\
    &\langle \bigtriangledown_{V_i}U_{\alpha},
   V_j\rangle=\frac{1}{2}\langle U_{\alpha},
   \big[V_j,V_i\big]\rangle_U
    =-\frac{1}{2}\langle U_{\alpha},\big[
   V_i,V_j]\rangle_U
   =-\frac{1}{2}\langle \eta j\big(U_{\alpha}\big)V_i,V_j\rangle_V,
\end{split}
\end{equation*} for any $V_i,V_j\in V$, $i\neq j$, and $U_{\alpha}\in U$. It is easy to verify that all the other combinations result as $0$.
From here we obtain that the connection is given by
\begin{equation}\label{eq:connection}\begin{split}
      &\bigtriangledown_{V_i}V_j=\frac{1}{2}
     \big[V_i,V_j\big],\\
      &\bigtriangledown_{U_{\alpha}}V_i
     =\bigtriangledown_{V_i}U_{\alpha}=-\frac{1}{2}\eta j\big(
    U_{\alpha}\big)V_i,\\
      &\bigtriangledown_{U_{\alpha}}U_{\beta}=0.\end{split}
\end{equation}

\section{Geodesics}\label{Geodesics}

In this section we compute geodesics passing through the identity with a given initial vector. By geodesic we mean 
a curve $\gamma\colon [0,1]\to G$ which satisfies the geodesic equation $\bigtriangledown_{\dot{\gamma}}{\dot{\gamma}}=0$. 
We define functions $t\mapsto v(t)$ and $t\mapsto u(t)$, such that $\gamma(t)=(v(t),u(t))$ and $\dot v(t)$ and $\dot u(t)$ are 
the projections of $\dot\gamma(t)$ onto $V$ and $U$ respectively.

Then in terms of global coordinates on $G$ one gets 
$\gamma(t)=\big(v_1(t),\ldots,v_n(t),u_1(t),\ldots,u_m(t)\big)$ and 
the derivative with respect to $t$ is
\begin{equation}\label{eq:derivative}
\begin{split}
      &\dot{\gamma}=\sum\limits_{i=1}^n\dot v_iV_i+\sum\limits_{\alpha=1}^m\dot u_{\alpha}U_{\alpha}
     =\sum\limits_{i=1}^n\dot{v}_i\Big(\partial_{v_i}
      +\frac{1}{2}\sum\limits_{\alpha=1}^m\big(\sum\limits_{j=1}^n\partial_{v_j}B^{\alpha}_{ji}\big)
     \partial_{u_{\alpha}}\Big)
     +\sum\limits_{\alpha=1}^m\dot u_{\alpha}\partial_{u_{\alpha}}\\
      &=\sum\limits_{i=1}^n\dot{v}_i\partial_{v_i}+\sum\limits_{\alpha=1}^m\Big(\dot u_{\alpha}
     +\frac{1}{2}\sum\limits_{i,j=1}^nv_jB^{\alpha}_{ji}\dot v_i\Big)\partial_{u_{\alpha}}
       =\dot v+\big(\dot u+\frac{1}{2}[\dot v,v]\big).
\end{split}
\end{equation}
Therefore,
\begin{equation}\label{eq:nabla}
       \bigtriangledown_{\dot{\gamma}}\dot{\gamma}=\bigtriangledown_{\dot{\gamma}}\dot v+\bigtriangledown_{\dot{\gamma}}\big(\dot u+\frac{1}{2}[\dot v,v]\big).
\end{equation}
Let us compute the first term in right hand side of \eqref{eq:nabla}
\begin{gather*}
       \bigtriangledown_{\dot{\gamma}}\dot v=\bigtriangledown_{\dot v+\dot u+\frac{1}{2}[\dot v,v]}\dot v
   =\bigtriangledown_{\dot v}\dot v+\bigtriangledown_{\dot u}\dot v+ \frac{1}{2}\bigtriangledown_{[\dot v,v]}\dot v
   =\ddot v- \frac{1}{2}\eta j(\dot u)\dot v-\frac{1}{4}\eta j([\dot v,v])\dot v\\
   =\ddot v-\frac{1}{2}\eta j\big(\dot u+\frac{1}{2}[\dot v,v]\big)\dot v.
\end{gather*}
Analogously, we calculate the second term
\begin{gather*}
     \bigtriangledown_{\dot{\gamma}}\big(\dot u+\frac{1}{2}[\dot v,v]\big)=\ddot u+\frac{1}{2}[\ddot v,v]-\frac{1}{2}\eta j\big(\dot u+\frac{1}{2}[\dot v,v]\big)\dot v.
\end{gather*}
Therefore,
\begin{gather*}
      \bigtriangledown_{\dot{\gamma}}\dot{\gamma}=\Big(\ddot v-\eta j\big(\dot u+\frac{1}{2}[\dot v,v]\big)\dot v\Big)+\big(\ddot u+\frac{1}{2}[\ddot v,v]\big)
\end{gather*}
and the geodesic $\gamma$ is characterised by the system of equations:
\begin{equation}\label{eq:sysgeod}
\begin{split}
       &\ddot v-\eta j\big(\dot u+\frac{1}{2}[\dot v,v]\big)\dot v=0,\\
       &\ddot u+\frac{1}{2}[\ddot v,v]=0.
\end{split}
\end{equation}
From the second equation of this sytem we can deduce that $\dot u+\frac{1}{2}[\dot v,v]=$constant. It is enough to describe geodesics 
which pass through $0$, since the left translation of every geodesic is again geodesic. Let us assume the following initial conditions for $\gamma$:
\begin{gather*}
     \gamma(0)=0,\qquad \dot{\gamma}(0)=\dot v^0+\dot u^0,
\end{gather*}
which implies that
\begin{gather*}
        v(0)=0,\qquad u(0)=0,\\
        \dot v(0)=\dot v^0,\qquad \dot u(0)=\dot u^0.
\end{gather*}
Then 
\begin{equation}\label{eq:uequation}
        \dot u+\frac{1}{2}[\dot v,v]=\dot u^0\quad \mbox{for all}\quad t\in[0,1]
\end{equation}
and the first equation of \eqref{eq:sysgeod} takes the form
\begin{equation}\label{eq:main}
     \ddot v-\eta j(\dot u^0)\dot v=0.
\end{equation}
Observe that this equation coincides with the one for sub-Riemannian geodesics when $\eta$ is an identity matrix in $V$ \cite{Kap2}.

To find a solution $v(t)$ of \eqref{eq:main} let us first recall some facts from linear algebra.

\subsection{Skew-symmetry with respect to semi-Riemannian metric}

Recall that $j(u)$ can be written as a real skew-symmetric $(n\times n)$ matrix: $j=-j^t$ if we fix a basis on $V$. 
Therefore, it has purely imaginary 
eigenvalues $\pm i\mu_k$, $\mu_k\in \mathbb{R}$, $k=1,\ldots,\frac{n}{2}$, 
since $n$ is an even number. It is possible to bring every skew-symmetric matrix to a block diagonal form 
by an orthogonal transformation. Thus, $j$ can be written in the form 
\begin{equation}\label{eq:jdecomp}
j=Q\tilde{j}Q^t,
\end{equation}
 where $\tilde{j}$ is a Murnaghan's canonical form of $j$:
\begin{gather*}
\tilde{j}=
\left(\begin{array}{cccc}
\begin{array}{cc}
 0& \mu_1\\
  -\mu_1&0\\
\end{array}
 & 0 & \ldots & 0\\
0 & \begin{array}{cc}
 0& \mu_2\\
  -\mu_2&0\\
\end{array} & \ldots & 0\\
\vdots & \vdots & \ddots & \vdots\\
0 & \ldots & 0 & \begin{array}{cc}
 0& \mu_{n/2}\\
  -\mu_{n/2}&0\\
\end{array}
\end{array}\right)
\end{gather*} and $Q$ is an orthogonal transformation. From the condition $j=-j^t$ we obtain also that $\tilde j=-\tilde j^t$. And 
since $j(u)$ has the property $i)$ of \eqref{eq:j}, all the eigenvalues of $j$ satisfy the condition $\mu_k^2=|u|^2$, 
$k=1,\ldots,\frac{n}{2}$.

Denote by $A$ the matrix $\eta j(u)$ from \eqref{eq:main}. The condition $ii)$ of \eqref{eq:j} yields that the matrix $A$ is skew-symmetric 
with respect to the scalar product on $V$ and it can be 
rewritten as $w^t\eta Av+(Aw)^t\eta v=0$, $\forall v,w\in V$, from where we obtain that 
\begin{equation}\label{eq:skew-symm}
   \eta A+A^t\eta=0,
\end{equation}
that is, $A^t=-\eta A\eta$ (skew-symmetric w.~r.~t. $\eta$). 

Every square matrix can be represented as a sum of symmetric and skew-symmetric ones, but in our case these matrices are of very 
special form, that follows from \eqref{eq:j-skew}:
\begin{equation}\label{eq:jview2}
    j=\left(%
\begin{array}{cc}
  \varPhi_p&\varPsi\\
  -\varPsi^t& \varPhi_q
\end{array}
    \right)=\left(%
\begin{array}{cc}
  \varPhi_p&0\\
  0& \varPhi_q
\end{array}
    \right)+
\left(%
\begin{array}{cc}
  0&\varPsi\\
  -\varPsi^t & 0
\end{array}
    \right)= A_1+A_2,
\end{equation}
where $\varPhi_p$ and $\varPhi_q$ are matrices of dimensions $p\times p$ and $q\times q$ 
respectively such that $\varPhi_p=-\varPhi_p^t$ and $\varPhi_q=-\varPhi_q^t$; $\varPsi$ is a $(p\times q)$-matrix and, therefore, matrix $A_2$ is skew-symmetric. 
In analogous way using \eqref{eq:skew-symm} we may represent the matrix $A=\eta j$ as follows
\begin{equation}\label{eq:Aview}
\begin{split}
    A=\left(%
\begin{array}{cc}
  -I_p&0\\
  0& I_q
\end{array}
    \right)
\left(%
\begin{array}{cc}
  \varPhi_p&\varPsi\\
  -\varPsi^t& \varPhi_q
\end{array}
    \right)=
\left(%
\begin{array}{cc}
  -\varPhi_p&-\varPsi\\
  -\varPsi^t& \varPhi_q
\end{array}
    \right)=\left(%
\begin{array}{cc}
  -\varPhi_p&0\\
  0& \varPhi_q
\end{array}
    \right)+
\left(%
\begin{array}{cc}
  0&-\varPsi\\
 -\varPsi^t & 0
\end{array}
    \right)\\=\tilde A_1+\tilde A_2,
\end{split}
\end{equation}
The essential distinction 
between matrices $j$ and $A=\eta j$  is that in the decomposition \eqref{eq:jview2} of $j$ 
both matrices $A_1$ and $A_2$ are skew-symmetric, 
but in decomposition \eqref{eq:Aview} of $A$ the matrix $\tilde A_1$ is skew-symmetric, but $\tilde A_2$ is symmetric.

The condition \eqref{eq:skew-symm} is the condition for $A$ to be skew-symmetric matrix with respect to 
$\langle\cdot,\cdot\rangle_V$ and is analogue of skew-symmetry condition \eqref{eq:j-skew} for positively definite metric. 
Using it now together with the condition $i)$ of \eqref{eq:j} we get
\begin{gather*}
    j^2(u)=\eta A\eta A=-A^t\eta\eta A=-A^tA=-|u|^2I_V.
\end{gather*}
From another hand,
\begin{gather*}
    j^2(u)=\eta A\eta A=-\eta AA^t\eta=-|u|^2I_V,
\end{gather*}
from where it follows that $AA^t=\eta|u|^2I_V\eta=|u|^2I_V$.
Therefore,
\begin{equation}\label{eq:Aat}
   AA^t=A^tA=|u|^2I_V.
\end{equation} 

This allows us to rewrite Kaplan's conditions \eqref{eq:j} as follows. Consider a mapping $A_u\colon V\to V$ (we will omit index $u$
 for the sake of simplicity), which 
is assigned to each $u\in U$ and satisfies 
the following properties:
\begin{equation*}
   \begin{split}
          &i') \;\; AA^t=A^tA=   |u|^2I_V,\\
          &ii') \;\;\langle Av,w\rangle_V=-\langle v,Aw\rangle_V.
   \end{split}
\end{equation*}
The calculations above show that the conditions $i)$ and $i')$ are equivalent; the condition $ii')$ is just a 
rewritten form of $ii)$ in terms of matrix $A$. Therefore, we get an independent definition of sub-semi-Riemannian $H$-type group.

Another effect of the property \eqref{eq:Aat} is that 
$A$ appears to have $n$ mutually orthogonal unit eigenvectors and is a diagonalisable matrix (over 
$\mathbb{C}$ in general). 
Observe that the 
matrix $A^2$ is not symmetric, but it 
satisfies the condition $(A^2)^t=\eta A^2\eta$ (symmetric with respect to $\eta$).

It is also worth mentioning that even though the matrix $A$ is real, 
the condition $A^t=A^*$ for adjoint matrix $A^*$  is not true in our situation as in the case of 
positively definite scalar product.
We define the adjoint matrix $A^{*,\eta}$ in our case by $\langle Av,w\rangle_V=\langle v,A^{*,\eta}w\rangle_V$. It 
has the following view $A^{*,\eta}=\eta A^t\eta$, i.~e. $A^{*,\eta}=-A$. Then the 
condition of $A$ being normal holds: $AA^{*,\eta}=A^{*,\eta}A=-A^2$.

Let us write the characteristic polynomial for the matrix $A$:
\begin{equation}
 \label{eq:char}
  p_A(\lambda)=\lambda^n+b_{n-1}\lambda^{n-1}+\ldots+b_2\lambda^2+b_1\lambda+b_0=0,
\end{equation}
where the coefficients $b_k$ in front of $\lambda^k$, $k=0,\ldots,n-1$, 
are computed by the known formula:
\begin{gather*}
  b_k=(-1)^k\sum\limits_{i=1}^n A^i_k, 
\end{gather*}
where by symbol $A^i_k$ we denote $(k\times k)$-dimensional principal minor of the matrix $A$, 
that is the determinant of the matrix 
obtained from the matrix $A$ by erasing its $n-k$ diagonal elements. 
For a $(n\times n)$-square matrix $A$ there are $n$ $(k\times k)$-dimensional principal minors for each $k=1,\ldots,n$. 
Notice that if $k$ is odd, $n-k$ is also odd.
\begin{lemma}
   All odd-dimensional principal minors of skew-symmetric w.~r.~t. $\eta$ matrix $A$ are zero.
\end{lemma}
\begin{proof}
We calculate
 $$\det (A^i_k)^t=\det A^i_k=\det\big(-\eta (A^i_k)^t\eta\big)=\det \big(-(A^i_k)^t\big)=
 \det(-I_k)\det(A^i_k)^t=(-1)^k\det (A^i_k)^t,$$
where $I_k$ is a $(k\times k)$-identity matrix. From here it follows that $\det A^i_k=0$ for odd $k$.
\end{proof}

Therefore, the characteristic polynomial \eqref{eq:char} receives the following form:
\begin{equation}\label{eq:char_new}
    p_A(\lambda)=\lambda^n+a_{\frac{n}{2}-1}\lambda^{n-2}+\ldots+a_1\lambda^2+a_0=0,
\end{equation}
where $a_k=b_{2k}$, $k=1,\ldots,\frac{n}{2}-1$, and $a_0=b_0=\det A$. Observe that $a_0\neq0$ since the dimension $n$ of the matrix $A$ is 
even. Therefore, $0$ is not an eigenvalue of $A$ and all eigenvalues of $A$ come in pairs $\pm\lambda_k$, $k=1,\ldots \frac{n}{2}$.

\begin{lemma}\label{prop1}
 If $\lambda$ is an eigenvalue for $A$ with eigenvector $v$, then $-\lambda$ is an eigenvalue for $A^t$ with eigenvector $\eta v$.
\end{lemma}
\begin{proof}
    Notice that $\langle v,w\rangle_V=\langle \eta v,\eta w\rangle_V$ for any two vectors $v$ and $w$ in $V$. We calculate
\begin{gather*}
    \langle Av,w\rangle_V=\langle \lambda v,w\rangle_V=\lambda \langle v,w\rangle_V=\lambda\langle\eta v,\eta w \rangle_V, 
\quad \forall w\in V.
\end{gather*}
But from another hand,
\begin{gather*}
    \langle Av,w\rangle_V=\langle -\eta A\eta^t v,w\rangle_V=-\langle A^t\eta v,\eta w\rangle_V,\quad \forall w\in V,
\end{gather*} 
yielding the desired result.
\end{proof}

The operator $A$ has $n$ in general complex eigenvalues and the corresponding complex eigenvectors. The following fact is known, 
see, for example, \cite{Vinberg}.
\begin{proposition}\label{complex}
    Let $A$ be a $(n\times n)$-matrix with real entries. If the equation $\det(A-\lambda I)=0$ has a complex solution $\lambda=\alpha+i\beta$, $\beta\neq0$, then there exist nonzero vectors $x$
 and 
    $y$ in $V$ such that
    $$Ax=\alpha x-\beta y,\qquad Ay=\alpha y+\beta x.$$
    In particular, $\spn\{x,y\}$ is a nontrivial invariant subspace for $A$.
\end{proposition}
Notice that if $\lambda = \alpha+i\beta$ is a solution of characteristic polynomial \eqref{eq:char_new} with multiplicity $l$, 
then its conjugate $\bar\lambda=\alpha-i\beta$ is also solution of \eqref{eq:char_new} with multiplicity $l$.
Due to this fact, lemma \ref{prop1}, proposition \ref{complex} and the fact that $A$ is diagonalisable over the complex space $V$, 
the following decomposition of $A$ is possible: $A=PDP^{-1}$, where $P$ is orthogonal matrix and 
$D$ is a block-diagonal matrix with ($2\times2$)-blocks of 
the following possible forms:
\begin{equation}\label{eq:real}
\left( \begin{array}{cc}
         \lambda_i&0\\
         0&-\lambda_i
    \end{array}\right),\quad i=1,\ldots s,
\end{equation}
\begin{equation}\label{eq:imaginary}
\left( \begin{array}{cc}
         i\nu_j&0\\
         0&-i\nu_j
    \end{array}\right),\quad j=s+1,\ldots r,
\end{equation}
and  $(4\times4)$-blocks of the form
\begin{equation}\label{eq:complex}
\left( \begin{array}{cc}
\begin{array}{cc}
         \alpha_k&\beta_k\\
         -\beta_k&\alpha_k
\end{array}&0\\
0&\begin{array}{cc}
-\alpha_k&-\beta_k\\
\beta_k&-\alpha_k
\end{array}
    \end{array}\right),\quad k=1,\ldots \frac{n}{4}-\frac{r}{2},
\end{equation}
where the eigenvalues are rearranged in the following order:
\begin{itemize}
\item $\pm\lambda_1,\ldots,\pm\lambda_s$, for some $s\in\{0,\ldots, \frac{n}{2}\}$, are real eigenvalues of $A$, \\
\item $\pm\lambda_{s+1}=\pm i\nu_{s+1},\ldots,\pm\lambda_r=\pm i\nu_r$, for some $r\in\{0,\ldots,\frac{n}{2}\}$, $r\geqslant s$,
 are purely imaginary eigenvalues of $A$, and \\
\item $\pm\lambda_{{r+1}}=\pm(\alpha_1+ i\beta_1)$, $\pm\bar\lambda_{r+1}$, $\ldots$, 
$\pm\lambda_{{\frac{n}{4}-\frac{r}{2}}}=\pm(\alpha_{\frac{n}{4}-\frac{r}{2}}+ i\beta_{\frac{n}{4}-\frac{r}{2}})$, 
$\pm\bar\lambda_{\frac{n}{4}-\frac{r}{2}}$ are complex 
eigenvalues of $A$ corresponding to $\frac{n}{4}-\frac{r}{2}$ blocks of the dimension $(4\times4)$. 
\end{itemize}
Notice that the integer $r$ is necessarily even. 
It is natural that $(4\times4)$-dimensional blocks associated with complex 
eigenvalues $\pm(\alpha\pm i\beta)$ can appear only if the dimension $n$ 
is bigger or equal to 4.

Observe that for real eigenvalues $\lambda_i$ the matrix  $\left(\begin{array}{cc}\lambda_i&0\\0&-\lambda_i\end{array}\right)$, 
$i=1,\ldots,s$,
 is a diagonalised form 
for the matrix $\left(\begin{array}{cc}0&\lambda_i\\\lambda_i&0\end{array}\right)$; and for purely imaginary eigenvalues 
$\lambda_j=i\nu_j$ the matrix $\left( \begin{array}{cc}i\nu_j&0\\
         0&-i\nu_j
    \end{array}\right)$  is a diagonalised form of $\left( \begin{array}{cc}
         0&\nu_j\\
         -\nu_j&0
    \end{array}\right)$, $j=s+1,\ldots,r$.

Notice that even though $A$ is skew-symmetric w.~r.~t. $\eta$, 
i.~e. the condition \eqref{eq:skew-symm} holds, 
the matrix $D$ does not have this property in general as long as $\eta$ is not identity matrix. 

\begin{lemma}
   All eigenvalues of $A$ have absolute value $|u|$.
\end{lemma}
\begin{proof}
 Let us calculate the product $AA^t=PDP^{-1}(P^t)^{-1}D^tP^t=PDD^tP^t$. Observe that the matrix $DD^t$ is a diagonal matrix with entries 
$|\lambda_k|^2$, $k=1,\ldots,\frac{n}{2}$. Then, using the property \eqref{eq:Aat}, we see that all eigenvalues of $A$, 
real and complex ones, have the same absolute 
values: $|\lambda_k|^2=|u|^2$, $k=1,\ldots,\frac{n}{2}$. 
\end{proof}

Taking into account previous comments about eigenvalues we can write that 
$$\lambda_1=\ldots=\lambda_s=\nu_{s+1}=\ldots=\nu_r=|u|,$$ $s,r\in\{0,\ldots,\frac{n}{2}\}$ and $\alpha_k^2+\beta_k^2=|u|^2$, 
$k=1,\ldots,\frac{n}{4}-\frac{r}{2}$. \\
Thus, we showed that the matrix $D$ is built from the the following cells:
\begin{gather*}
\left(%
\begin{array}{cc}
  0&|u|\\
  |u|&0
\end{array}
    \right), \qquad \left(%
\begin{array}{cc}
  0&|u|\\
  -|u|&0
\end{array}
    \right)\quad \text{ or }\quad
\left(%
\begin{array}{cccc}
  \alpha&\beta&0&0\\
  -\beta&\alpha&0&0\\
  0&0&-\alpha&-\beta\\
  0&0&\beta&-\alpha
\end{array}\right),
\end{gather*} where $\alpha^2+\beta^2=|u|^2$.

The appearance of one or another cell of these types in the matrix $D$ 
depends on the index $p$ 
of the metric $g$. 

Below we show that the number $s$ of pairs of real eigenvalues of $A$ is equal to $1$ in the case of odd index $p$ of the metric $\eta$, 
and $s=0$ in case of $p$ is even. In other words, if index $p$ is even, 
then there are no real eigenvalues of $A$, and if $p$ is odd, there are only two real eigenvalues that are of the form $\pm|u|$.

{\bf Note}.
Notice that $\det A=\det D=\big(-|u|^2\big)^s\big(|u|^2\big)^{r-s}\big(|u|^2\big)^{\frac{n}{2}-r}=(-1)^s|u|^n$.
But on the other hand from \eqref{eq:jdecomp} we have $\det A=\det \big(\eta j\big)= \det\big(\eta Q\tilde j Q^{-1}\big)=(-1)^p|u|^n$. 
Thus, if $p$ is odd, then $s$ is odd too, and if $p$ is even, then $s$ is also even. Notice also that $r-s$ has the same parity as $s$.

{\bf Remark}.
We restrict our consideration only to the case when the index of the metric $p$ is less or equal to $\frac{n}{2}$, because in an 
opposite situation we can study the case with the metric $-\eta$, which is analogous.

\subsection{Matrices $\bf{j_1,\ldots,j_m}$}\label{matrices}

We remind that matrix $j(u)$ does not depend on the metric $\eta$ and is the same as for sub-Riemannian $H$-type group, i.~e. can 
be written via endomorphisms $j_1,\ldots,j_m$ represented by structure constants \eqref{eq:B} in the following form:
$$j(u)=u_1j_1+\ldots +u_mj_m,\quad u=(u_1,\ldots,u_m)\in U.$$
Each of $j_{\alpha}$ is, what is called in linear algebra, a one-column and one-row matrix with $1$ at the intersection of $i$-th row and $j$-th column, 
$-1$ at the intersection of $j$-th row and $i$-th column, for $i\neq j$ and each $i,j<n$ and zeros everywhere else (see, for example, 
matrices $j_{\alpha}$, $\alpha=1,\ldots,7$ from Section \ref{Example}); 
and satisfies the following properties
\begin{equation}\label{eq:proper}
  j^t_{\alpha}=-j_{\alpha},\quad j_{\alpha}^2=-I_V,\quad j_{\alpha}j_{\beta}=-j_{\beta}j_{\alpha},\quad \alpha\neq\beta.
\end{equation}
The block-diagonal $(n\times n)$-matrix consisting of blocks of type $\left(\begin{array}{cc}
0&\pm1\\\mp1&0\end{array}\right)$ is one of them, and let us denote it by $j_1$ for simplicity.
Since $\Cl(U,m)$ by definition is the associative algebra freely generated by $m$ anti-commuting square roots of minus unity, one can 
associate matrices $j_1,\ldots,j_m$ with square roots of $-I_V$. Their linear combination represents the corresponding endomorphism $j(u)$. 
The matrices $j_{\alpha}$ have $n$ eigenvalues which come in pairs $\pm i$ and, therefore, $\det j_{\alpha}=1$ for each 
$\alpha=1,\ldots,m$.

We remind that $j(u)$ is a $(n\times n)$-matrix with the following properties:
\begin{itemize}
 \item $j^t(u)=-j(u)$;\\
 \item $j^2(u)=-|u|^2I_V$;\\
 \item $j^{-1}(u)=-\frac{1}{|u|^2}I_V$;\\
 \item the eigenvalues of $j(u)$ are all of the form $\pm i|u|$;\\
 \item the determinant of $j(u)$ $\det j(u)=|u|^n$.
\end{itemize}
From here we can catch out the following properties of matrices $\eta j_{\alpha}$, $\alpha=1,\ldots,m$.
\begin{lemma}\label{pro1}
  Each matrix $\eta j_{\alpha}$ is skew-symmetric w.~r.~t. $\eta$, i.~e. $\eta j_{\alpha}+j_{\alpha}^t\eta=0$.
\end{lemma}
\begin{proof}
   The proof is the same as for an analogous statement \eqref{eq:skew-symm} for the matrix $A=\eta j(u)$ in the previous 
sub-section.
\end{proof}
Analogous argument as for matrix $A$ can be applied to show that all eigenvalues of matrices $\eta j_{\alpha}$, $\alpha=1,\ldots,m$, 
 have absolute value $1$ and come in pairs $\pm\lambda_k$, $k=1,\ldots,\frac{n}{2}$.
\begin{lemma}\label{pro2}
  All eigenvalues for matrix $\eta j_{\alpha}$ are of the form $\pm1$ or $\pm i$, for each $\alpha=1,\ldots,m$.
\end{lemma}
\begin{proof}
 If $a+ib$ is an eigenvalue for $\eta j_{\alpha}$, then by proposition \ref{complex} there exist vectors $x$ and $y$ 
such that $\eta j_{\alpha}x=a x-by$ and $\eta j_{\alpha}y=ay+b x$. Thus, on the space spanned by vectors $x$ and $y$ 
the linear transformation $\eta j_{\alpha}$ has the matrix
\begin{gather*}
  \left(\begin{array}{cc}
         a&b\\
         -b&a
        \end{array}
  \right),
\end{gather*}
which is possible only if $a=0$ or $b=0$, since $\eta j_{\alpha}$ is a 1-row and 1-column matrix. Due to the fact that 
all eigenvalues of $\eta j_{\alpha}$ have absolute value $1$, they should be of the form $\pm1$ or $\pm i$.
\end{proof}

\begin{lemma}\label{pro3}
   Matrices $\eta j_{\alpha}$, $\alpha=1,\ldots,m$, anti-commute with each other.
\end{lemma}
\begin{proof} Using lemma \ref{pro1} and anti-commutativity of matrices $j_{\alpha}$ we calculate for $\alpha\neq\beta$:
    \begin{gather*}\eta j_{\alpha}\eta j_{\beta}=-\eta j_{\alpha} j_{\beta}^t\eta=\eta j_{\alpha} j_{\beta}\eta=
-\eta j_{\beta}j_{\alpha}\eta=\eta j_{\beta}j_{\alpha}^t\eta=-\eta j_{\beta}\eta j_{\alpha}.\end{gather*}
\end{proof}

\begin{lemma}\label{pro4} If the index $p$ of metric $\eta$ is odd, then each matrix $\eta j_{\alpha}$ 
has at least two real eigenvalues $\pm 1$. 
\end{lemma}
\begin{proof}
   The proof follows from the fact that $\det\big(\eta j_{\alpha}\big)=(-1)^p=-1$ for the odd $p$. From another hand, 
   the determinant of matrix $\eta j_{\alpha}$ is equal to product of all eigenvalues of $\eta j_{\alpha}$ that is 
   $\underbrace{(-1^2)\cdot\ldots\cdot (-1^2)}_{s\;times}
\underbrace{(-i^2)\cdot\ldots\cdot (-i^2)}_{\frac{n}{2}-s\;times}$, 
where $0\leqslant s\leqslant \frac{n}{2}$ is an integer. 
This product equals to $-1$ if and only if there is an odd number $s$ of pairs of real eigenvalues $\pm1$.
\end{proof}

Note that matrix $\eta j_1$ has only 2 real eigenvalues $\pm1$ if $p$ is odd.
For $p$ -- even there exists at least one matrix from $\eta j_1,\ldots,\eta j_m$ 
which does not have real eigenvalues, for example, matrix $\eta j_1$.

Notice also that if the matrix $\eta j_{\alpha}$ has eigenvalues of the form $\pm1$ and/or $\pm i$, then 
the matrix $u_{\alpha}\eta j_{\alpha}$ has eigenvalues of the form $\pm u_{\alpha}$ and/or $\pm i u_{\alpha}$ 
respectively, where $u_{\alpha}$ is a real number.

\begin{lemma}\label{pro5}
  If $\pm1$ are eigenvalues for all of the matrices $\eta j_1,\ldots,\eta j_m$, then $\pm|u|$ are eigenvalues for $\eta j(u)$.
\end{lemma}
\begin{proof}
We will need the following proposition.
\begin{proposition}
    Given two anti-commuting $(n\times n)$-matrices $A$ and $B$ with real eigenvalues $\pm a$ and $\pm b$ respectively, 
    their sum $A+B$ 
    has two of its $n$ eigenvalues of the form $\pm\sqrt{a^2+b^2}$.
\end{proposition}
{\it Proof of the proposition}.
Let $V_a$ be the $a$-eigenspace of the matrix $A$. Then by anti-commutativity we find $BV_a\subseteq V_{-a}$. Notice that matrices 
$A$ and $B^2$ commute and, therefore, are simultaneously diagonalizable. So there exists an eigenvector $v\in V_a$ with the property 
$B^2v=b^2v$. Then notice that for the vector $w=\frac{1}{b}Bv$ we have the following relations 
\begin{gather*}
    Av=av,\qquad Bv=bw,\\
    Aw=-aw,\qquad Bw=bv.
\end{gather*} From here it follows that on the space spanned by vectors $v$ and $w$ the linear transformation $A+B$ has the matrix 
\begin{gather*}
    \left(
      \begin{array}{cc}
          a&b\\
          b&-a
      \end{array}
    \right),
\end{gather*}
which has eigenvectors with eigenvalues $\pm\sqrt{a^2+b^2}$.
\qed

Easy induction shows that the same argument applies to the sum of $m$ matrices: if $\{A_{\alpha}\}_{\alpha=1}^m$  is a set of
 anti-commuting matrices and each matrix $A_{\alpha}$ has two real eigenvalues 
$\pm a_{\alpha}$, $\alpha=1,\ldots,m$, then the matrix $A_1+\ldots+A_m$ has two real eigenvalues of the form 
$\pm\sqrt{a_1^2+\ldots+a_m^2}$.

We obtain the statement of Lemma \ref{pro5} by taking matrices $\eta j_{\alpha}=A_{\alpha}$ and numbers 
$u_{\alpha}=a_{\alpha}$, $\alpha=
1,\ldots,m$.
\end{proof}

Thus, Lemmas \ref{pro1}-\ref{pro5} can be summed up in the following theorem.
\begin{theorem}
  If the index $p$ of the metric $\eta$ is odd, then there are two real eigenvalues $\pm|u|$ for the matrix $\eta j(u)$. If $p$ is even, then 
there are no real eigenvalues for $\eta j(u)$ at all.
\end{theorem}
\begin{proof}
 In order to have 2 real eigenvalues $\pm|u|$ for the matrix $A=\eta j(u)$ all matrices $\eta j_{\alpha}$ should have $\pm1$ 
as eigenvalues, which is possible only if $p$ is odd. There are only two real eigenvalues for $A$, since matrix $\eta j_1$ has only two real eigenvalues.
 If $p$ is even, matrix $\eta j_1$ does not have real eigenvalues, and, therefore, 
$A$ does not have real eigenvalues either.
\end{proof}

Further idea is to construct $\frac{n}{2}$ 2-dimensional invariant subspaces of the operator $A\colon V\to V$. 
Remind that $V$ has a basis consisting of $p$ timelike and $q$ spacelike vectors.
There are two possible cases here depending on the index $p$ being odd or even.
In case $p$ is odd we have
\begin{gather*}
    V=W_0\oplus W_1\oplus \ldots\oplus W_{\frac{p-1}{2}}\oplus \tilde W_1\oplus\ldots 
    \oplus \tilde W_{\frac{q-1}{2}},
\end{gather*}
where the restriction of the metric $g$ on the subspace $W_0$ $\left.g\right|_{W_0}$ is Lorentzian, that is it has index 1, 
$\left.g\right|_{W_i}$, $i=1,\ldots,\frac{p-1}{2}$, is a negatively definite metric and $\left.g\right|_{\tilde W_j}$, $j=1,\ldots,
\frac{q-1}{2}$, is a positively definite metric.
In case of even $p$ we have
\begin{gather*}
    V=W_1\oplus \ldots\oplus W_{\frac{p}{2}}\oplus \tilde W_1\oplus\ldots 
    \oplus \tilde W_{\frac{q}{2}},
\end{gather*}
where $\left.g\right|_{W_i}$, $i=1,\ldots,\frac{p}{2}$, is a negatively definite metric and $\left.g\right|_{\tilde W_j}$, 
$j=1,\ldots, \frac{q}{2}$, is a positively definite metric. 
The basis of the subspace $W_0$ consists of one timelike and one spacelike unit vectors, the basis of subspaces $W_i$, $i\neq0$,
 consists of 
timelike unit vectors and the basis of subspaces $\tilde W_j$ consists of spacelike unit vectors.

\subsection{Formulae for geodesics}

Let us return to the system of equations \eqref{eq:sysgeod}. It is equivalent to the system in new basis 
$\tilde v=Pv\in \mathbb{R}^n$
\begin{equation}\label{eq:system}
\ddot{ \tilde v}=D\dot {\tilde v},
\end{equation}
where $D=P^{-1}AP$ is a block-type matrix with blocks of the form \eqref{eq:real}--\eqref{eq:complex} and $P$ is orthogonal matrix.

It is convinient to renumerate coordinates in $\tilde v=(\tilde v_1,\ldots,\tilde v_n)$ in the following way:\\ 
$\tilde v=(\tilde v_{s_1},
\tilde v_{s_2},\tilde v_{{s+1}_1},\tilde v_{{s+1}_2},\ldots,\tilde v_{r_1},\tilde v_{r_2},\tilde v_{{r+1}_1},\tilde v_{{r+1}_2},
\tilde v_{{r+1}_3},\tilde v_{{r+1}_4},\ldots,\tilde v_{{\frac{n}{4}+
\frac{r}{2}}_1},\tilde v_{{\frac{n}{4}+\frac{r}{2}}_2},\tilde v_{{\frac{n}{4}+\frac{r}{2}}_3},\tilde v_{{\frac{n}{4}+\frac{r}{2}}_4})$,
where first coordinates $\tilde v_{s_1},\tilde v_{s_2}$ correspond to real eigenvalues $\pm|u|$ (if any), then coordinates 
$\tilde v_{j_1},\tilde v_{j_2}$, $j=s+1,\ldots,r$, corresponding to purely imaginary 
eigenvalues $\pm i|u|$ (if any), and then coordinates 
$\tilde v_{r+k_1},\tilde v_{r+k_2},
\tilde v_{r+k_3},\tilde v_{r+k_4}$, $k=1,\ldots,\frac{n}{4}-\frac{r}{2}$ 
corresponding to complex eigenvalues $\pm(\alpha_k\pm i\beta_k)$, $\alpha_k^2+\beta_k^2=|u|^2$, $k=1,\ldots\frac{n}{4}-\frac{r}{2}$ 
(if $\frac{r}{2}<\frac{n}{4}$). Here lowest index in $\tilde v_{i_1},\tilde v_{i_2},\tilde v_{i_3},\tilde v_{i_4}$ denotes the number of 
the coordinate in $\tilde v$ associated with the eigenvalues 
$\lambda_i,-\lambda_i,\bar{\lambda}_i,-\bar{\lambda}_i$ respectively.

Therefore, the system \eqref{eq:system} splits into $\frac{n}{2}$ independent systems each in 2 variables
\begin{equation}\label{eq:sysnew}
  \begin{split}
     & \begin{cases}   
           \ddot{\tilde v}_{s_1}=|u|\dot{\tilde v}_{s_2},\\
           \ddot{\tilde v}_{s_2}=|u|\dot{\tilde v}_{s_1},
        \end{cases}\\
      & \begin{cases}   
           \ddot{\tilde v}_{j_1}=|u|\dot{\tilde v}_{j_2},\\
           \ddot{\tilde v}_{j_2}=-|u|\dot{\tilde v}_{j_1},\qquad j=s+1,\ldots,r,
        \end{cases}\\
     & \begin{cases}   
           \ddot{\tilde v}_{r+k_1}=\;\;\,\alpha_k\dot{\tilde v}_{r+k_1}+\beta_k\dot{\tilde v}_{r+k_2},\\
           \ddot{\tilde v}_{r+k_2}=-\beta_k\dot{\tilde v}_{r+k_1}+\alpha_k\dot{\tilde v}_{r+k_2},\\
       \end{cases}\\
     & \begin{cases} 
           \ddot{\tilde v}_{r+k_3}=-\alpha_k\dot{\tilde v}_{r+k_3}-\beta_k\dot{\tilde v}_{r+k_4}, \qquad \alpha_k^2+\beta_k^2=|u|^2,\\
           \ddot{\tilde v}_{r+k_4}=\;\;\;\beta_k\dot{\tilde v}_{r+k_3}-\alpha_k\dot{\tilde v}_{r+k_4},\qquad k=1,\ldots,\frac{n}{4}-\frac{r}{2}.
        \end{cases}\\
  \end{split}
\end{equation}

There are $n$ equations and $n$ variables in \eqref{eq:sysnew}.
We are looking for the solution $\tilde v=\tilde v(t)\in V$, $t\in [-\infty,+\infty]$, $v(0)=0$ and with initial velocity 
$\dot {\tilde v}(0)=\dot {\tilde v}^0$. 
From the system \eqref{eq:sysnew} we get
\begin{equation*}\begin{split}
        \begin{cases}   
           \dot{\tilde v}_{s_1}=-c_{s_1} e^{-|u|t}+c_{s_2} e^{|u|t}, \qquad \dot{\tilde v}_{s_1}^0=-c_{s_1}+c_{s_2},\\
           \dot{\tilde v}_{s_2}=\;\;\;c_{s_1} e^{-|u|t}+c_{s_2} e^{|u|t},\qquad \dot{\tilde v}_{s_2}^0=c_{s_1}+c_{s_2},
        \end{cases}\\
         \begin{cases}   
           \dot{\tilde v}_{j_1}=c_j\sin(|u|t)+c_{j_2}\cos(|u|t),\qquad \dot{\tilde v}_{j_1}^0=c_{j_2}, \\
           \dot{\tilde v}_{j_2}=c_j\cos(|u|t)-c_{j_2}\sin(|u|t),\qquad \dot{\tilde v}_{j_2}^0=c_{j_1},
                   \end{cases}\\
\qquad j=s+1,\ldots,r,\end{split}
\end{equation*}
\begin{equation*}\begin{split}
         \begin{cases}   
           \dot{\tilde v}_{r+k_1}= e^{\alpha_kt}\big(-c_{r+k_1}\cos\beta_kt+c_{r+k_2}\sin\beta_kt\big),\qquad  \dot{\tilde v}_{r+k_1}^0=-c_{r+k_1}\\
           \dot{\tilde v}_{r+k_2}=e^{\alpha_kt}\big(c_{r+k_1}\sin\beta_kt+c_{r+k_2}\cos\beta_kt\big),\qquad  \dot{\tilde v}_{r+k_2}^0=c_{r+k_2},
        \end{cases}\\ 
         \begin{cases}   
           \dot{\tilde v}_{r+k_3}= e^{-\alpha_kt}\big(c_{r+k_3}\cos\beta_kt-c_{r+k_4}\sin\beta_kt\big),\qquad  \dot{\tilde v}_{r+k_3}^0=c_{r+k_3}\\
           \dot{\tilde v}_{r+k_4}=e^{-\alpha_kt}\big(c_{r+k_3}\sin\beta_kt+c_{r+k_4}\cos\beta_kt\big),\qquad  \dot{\tilde v}_{r+k_4}^0=c_{r+k_4},
        \end{cases}\\     
\qquad k=1,\ldots,\frac{n}{4}-\frac{r}{2},\end{split}
\end{equation*}

It follows that 
\begin{equation*}\begin{split}
        \begin{cases}   
           \dot{\tilde v}_{s_1}(t)=\dot{\tilde v}_{s_1}^0\cosh(|u|t)+\dot{\tilde v}_{s_2}^0\sinh(|u|t),\\
           \dot{\tilde v}_{s_2}(t)=\dot{\tilde v}_{s_2}^0\cosh(|u|t)+\dot{\tilde v}_{s_1}^0\sinh(|u|t),
        \end{cases}\\
        \begin{cases}   
           \dot{\tilde v}_{j_1}(t)=\dot{\tilde v}_{j_2}^0\sin(|u|t)+\dot{\tilde v}_{j_1}^0\cos(|u|t),\\
           \dot{\tilde v}_{j_2}(t)=\dot{\tilde v}_{j_2}^0\cos(|u|t)-\dot{\tilde v}_j^0\sin(|u|t),
           \end{cases}\\  \qquad j=s+1,\ldots,r,\end{split}
\end{equation*}
\begin{equation*}\begin{split}
        \begin{cases}   
           \dot{\tilde v}_{r+k_1}(t)=\dot{\tilde v}_{r+k_2}^0\sinh \alpha_kt\sin\beta_kt
                              +\dot{\tilde v}_{r+k_1}^0\cosh \alpha_kt\cos\beta_kt\\
                              \qquad\qquad+\dot{\tilde v}_{r+k_1}^0\sinh \alpha_kt\cos\beta_kt
                              +\dot{\tilde v}_{r+k_2}^0\cosh \alpha_kt\sin\beta_kt,\\
           \dot{\tilde v}_{r+k_2}(t)=-\dot{\tilde v}_{r+k_1}^0\sinh \alpha_kt\sin\beta_kt
                                   +\dot{\tilde v}_{r+k_2}^0\cosh \alpha_kt\cos\beta_kt\\
                                  \qquad\qquad +\dot{\tilde v}_{r+k_2}^0\sinh \alpha_kt\cos\beta_kt
                                   -\dot{\tilde v}_{r+k_1}^0\cosh \alpha_kt\sin\beta_kt,
        \end{cases}\end{split}
\end{equation*}
\begin{equation*}\begin{split}
        \begin{cases}   
           \dot{\tilde v}_{r+k_3}(t)=\dot{\tilde v}_{r+k_4}^0\sinh \alpha_kt\sin\beta_kt
                              +\dot{\tilde v}_{r+k_3}^0\cosh \alpha_kt\cos\beta_kt\\
                              \qquad\quad-\dot{\tilde v}_{r+k_3}^0\sinh \alpha_kt\cos\beta_kt
                              -\dot{\tilde v}_{r+k_4}^0\cosh \alpha_kt\sin\beta_kt,\\
           \dot{\tilde v}_{r+k_4}(t)=-\dot{\tilde v}_{r+k_3}^0\sinh \alpha_kt\sin\beta_kt
                                   +\dot{\tilde v}_{r+k_4}^0\cosh \alpha_kt\cos\beta_kt\\
                                  \qquad\qquad -\dot{\tilde v}_{r+k_4}^0\sinh \alpha_kt\cos\beta_kt
                                   +\dot{\tilde v}_{r+k_3}^0\cosh \alpha_kt\sin\beta_kt,              
        \end{cases}    \\
 k=1,\ldots,\frac{n}{4}-\frac{r}{2}.\end{split}
\end{equation*}
Finally, we obtain the solution $\tilde v(t)$ to the system \eqref{eq:sysnew}
\begin{equation}\label{eq:sollor}
        \begin{cases}   
           \tilde v_{s_1}(t)=\frac{1}{|u|}\big(\dot{\tilde v}_{s_1}^0\sinh(|u|t)+\dot{\tilde v}_{s_2}^0\cosh(|u|t)-\dot{\tilde v}_{s_2}^0\big),\\
           \tilde v_{s_2}(t)=\frac{1}{|u|}\big(\dot{\tilde v}_{s_2}^0\sinh(|u|t)+\dot{\tilde v}_{s_1}^0\cosh(|u|t)-\dot{\tilde v}_{s_1}^0\big),
        \end{cases}
\end{equation}
\begin{equation}\label{eq:solriem}
        \begin{cases}   
           \tilde v_{j_1}(t)=\frac{1}{|u|}\big(\dot{\tilde v}_{j_1}^0\sin(|u|t)-\dot{\tilde v}_{j_2}^0\cos(|u|t)+\dot{\tilde v}_{j_2}^0\big),\\
           \tilde v_{j_2}(t)=\frac{1}{|u|}\big(\dot{\tilde v}_{j_2}^0\sin(|u|t)+\dot{\tilde v}_{j_1}^0\cos(|u|t)-\dot{\tilde v}_{j_1}^0\big),\\        
\end{cases}\\
\qquad j=s+1,\ldots,r,
\end{equation}
\begin{equation}\label{eq:solcomplex1}
\begin{split}
        \begin{cases}   
           \tilde v_{r+k_1}(t)= A_k\sinh \alpha_kt\sin\beta_kt
                           +B_k\cosh \alpha_kt\cos\beta_kt\\
                           \;\;\qquad+B_k\sinh \alpha_kt\cos\beta_kt
                           +A_k\cosh \alpha_kt\sin\beta_kt-B_k,\\
           \tilde v_{r+k_2}(t)=-B_k\sinh \alpha_kt\sin\beta_kt
                           +A_k\cosh \alpha_kt\cos\beta_kt\\
                           \qquad\qquad\;\:+A_k\sinh \alpha_kt\cos\beta_kt
                           -B_k\cosh \alpha_kt\sin\beta_kt-A_k, 
        \end{cases}\\ 
 A_k=\frac{1}{|u|^2}(\alpha_k\dot{\tilde v}_{r+k_2}^0+\beta_k\dot{\tilde v}_{r+k_1}^0),
\quad B_k=\frac{1}{|u|^2}(\alpha_k\dot{\tilde v}_{r+k_1}^0-\beta_k\dot{\tilde v}_{r+k_2}^0),
\end{split}
\end{equation}
\begin{equation}\label{eq:solcomplex2}
\begin{split}
        \begin{cases}   
           \tilde v_{r+k_3}(t)=-C_k\sinh \alpha_kt\sin\beta_kt
                              -D_k\cosh \alpha_kt\cos\beta_kt\\
                              \qquad\qquad\;\,+D_k\sinh \alpha_kt\cos\beta_kt
                              +C_k\cosh \alpha_kt\sin\beta_kt+D_k,\\
           \tilde v_{r+k_4}(t)=D_k\sinh \alpha_kt\sin\beta_kt
                               -C_k\cosh \alpha_kt\cos\beta_kt\\
                                  \qquad\quad\;\;\; +C_k\sinh \alpha_kt\cos\beta_kt
                                   -D_k\cosh \alpha_kt\sin\beta_kt+C_k,              
        \end{cases}    \\
 C_k=\frac{1}{|u|^2}(\alpha_k\dot{\tilde v}_{r+k_4}^0+\beta_k\dot{\tilde v}_{r+k_3}^0),
\quad D_k=\frac{1}{|u|^2}(\alpha_k\dot{\tilde v}_{r+k_3}^0-\beta_k\dot{\tilde v}_{r+k_4}^0),\\
 k=1,\ldots,\frac{n}{4}-\frac{r}{2}.
\end{split}
\end{equation}
A curve $\tilde v(t)$ has the same causal type as the initial velocity vector $\dot{\tilde v}^0$ (\cite{Beem}).
\begin{lemma}
    The projection of the curve $\tilde v(t)$ onto the $(\tilde v_{s_1},\tilde v_{s_2})$-plane, is a brunch of the hyperbola 
    with the canonical equation 
    \begin{equation}\label{eq:hyperbola}
             \Big(\tilde v_{s_1}+\frac{\dot{\tilde v}^0_{s_2}}{|u|}\Big)^2-
             \Big(\tilde v_{s_2}+\frac{\dot{\tilde v}^0_{s_1}}{|u|}\Big)^2=\frac{-(\dot{\tilde v}^0_{s_1})^2+(\dot{\tilde v}^0_{s_2})^2}{|u|^2}.
    \end{equation}
\end{lemma}
\begin{proof}
From \eqref{eq:sollor} we calculate that $\tilde v_{s_1}^2(t)-\tilde v_{s_2}^2(t)=
\frac{2\big((\dot{\tilde v}^0_{s_1})^2-(\dot{\tilde v}^0_{s_2})^2\big)}{|u|^2}(\cosh|u|t-1)$. 
This expression can be rewritten as stated in \eqref{eq:hyperbola}.
\end{proof}
\begin{lemma}
     The projection of the curve $\tilde v(t)$ onto the $(\tilde v_{j_1},\tilde v_{j_2})$-plane, $j=s+1,\ldots,r$, is a circle with the 
     center $\Big(-\frac{\dot{\tilde v}^0_{j_2}}{|u|},\frac{\dot{\tilde v}^0_{j_1}}{|u|}\Big)$ of radius 
     $\frac{\sqrt{(\dot{\tilde v}^0_{j_1})^2+(\dot{\tilde v}^0_{j_2})^2}}{|u|}$.
\end{lemma}
\begin{proof}
     From \eqref{eq:solriem} we get $\tilde v_{j_1}^2(t)+\tilde v_{j_2}^2(t)=
     \frac{4\big((\dot{\tilde v}^0_{j_1})^2+(\dot{\tilde v}^0_{j_2})^2\big)}{|u|^2}\sinh^2\frac{|u|t}{2}$. This expression leads to
    $$\Big(\tilde v_{j_1}+\frac{\dot{\tilde v}^0_{j_2}}{|u|}\Big)^2+\Big(\tilde v_{j_2}-\frac{\dot{\tilde v}^0_{j_1}}{|u|}\Big)^2
    =\frac{(\dot{\tilde v}^0_{j_1})^2+(\dot{\tilde v}^0_{j_2})^2}{|u|^2}.$$
\end{proof}

\begin{lemma}
    The projection of the curve $\tilde v(t)$ onto the $(\tilde v_{r+k_1},\tilde v_{r+k_2})$-plane, $k=1,\ldots,\frac{n}{4}-\frac{r}{2}$, is a logarithmic spiral 
    with the equation 
    \begin{equation}\label{eq:spiral1}
             \tilde v_{r+k_1}^2(t)+\tilde v_{r+k_2}^2(t)=\frac{1}{|u|^2}\Big((\dot{\tilde v}^0_{r+k_1})^2+(\dot{\tilde v}^0_{r+k_2})^2\Big)e^{2\alpha_kt}.
    \end{equation}
\end{lemma}
\begin{proof}
 From \eqref{eq:solcomplex1} we obtain
$$\tilde v_{r+k_1}^2(t)+\tilde v_{r+k_2}^2(t)=\frac{1}{|u|^2}\big(A_k^2+B_k^2\big)e^{2\alpha_kt}.$$
Substituting $A_k$ and $B_k$ into $A_k^2+B_k^2$ we get the desired formula \eqref{eq:spiral1}.
\end{proof}
\begin{figure}[ht]
\centering \scalebox{0.7}{\includegraphics{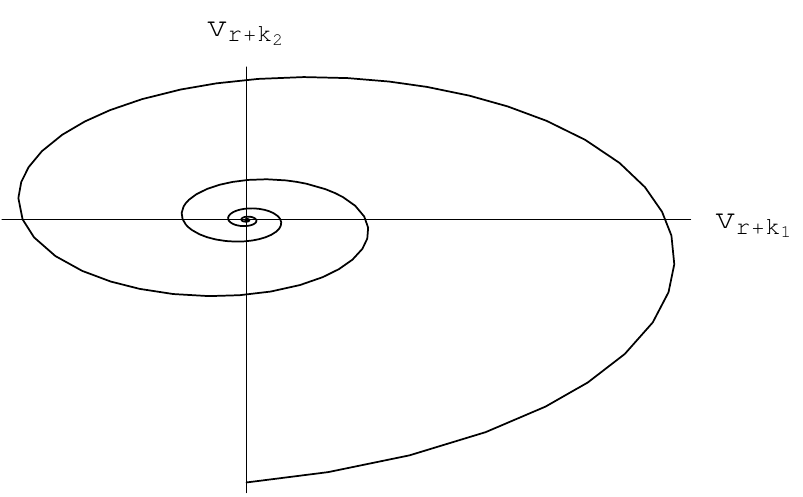}}
\caption[]{Projection of geodesic onto $(\tilde v_{r+k_1},\tilde v_{r+k_2})$-plane}
\end{figure}
Analogously using \eqref{eq:solcomplex2} can be shown the next lemma.
\begin{lemma}
    The projection of the curve $\tilde v(t)$ onto the $(\tilde v_{r+k_3},\tilde v_{r+k_4})$-plane, $k=1,\ldots,\frac{n}{4}-\frac{r}{2}$, is a logarithmic spiral 
    with the equation 
    \begin{equation}\label{eq:spiral2}
             \tilde v_{r+k_3}^2(t)+\tilde v_{r+k_4}^2(t)=\frac{1}{|u|^2}\Big((\dot{\tilde v}^0_{r+k_3})^2+(\dot{\tilde v}^0_{r+k_4})^2\Big)e^{-2\alpha_kt}.
    \end{equation}
\end{lemma}
\begin{figure}[ht]
\centering \scalebox{0.7}{\includegraphics{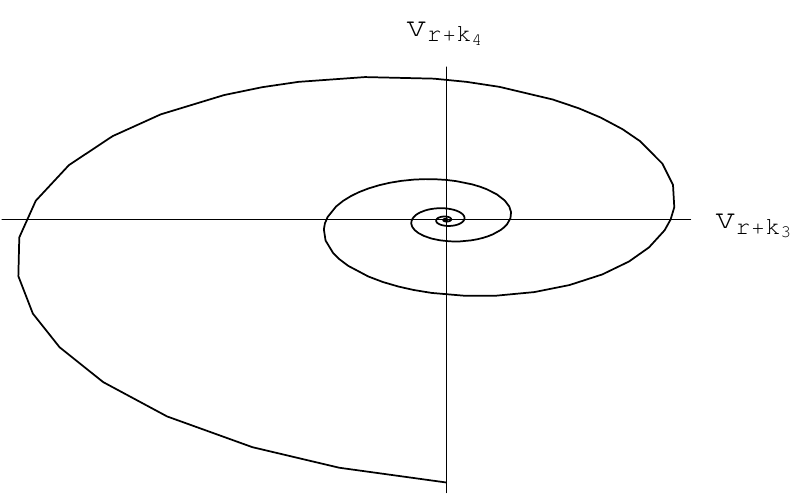}}
\caption[]{Projection of geodesic onto $(\tilde v_{r+k_3},\tilde v_{r+k_4})$-plane}
\end{figure}

\subsection{Solution in a matrix form}

Let us rewrite the matrix $D$ \eqref{eq:real}-\eqref{eq:complex} in the following form $\tilde D=\tilde P D\tilde P^t$, where the 
matrix $\tilde P$ is obtained from the matrix $P$ by rearranging lines in the blocks of type \eqref{eq:complex} and
\begin{equation}\label{eq:tildeD}
    \tilde D=\left(
      \begin{array}{cccc}
           |u|D_1&0&0&0\\
           0&|u|D_2&0&0\\
           0&0&|u|^2D_3&0\\
           0&0&0&|u|^2D_4
      \end{array}
\right),
\end{equation}
where $D_1$ is a $(2s\times 2s)$-matrix  
$\left(\begin{array}{cc} 0&1\\1&0\end{array}\right)$,
$D_2$ is a block-diagonal $\big(2(r-s)\times2(r-s)\big)$-matrix consisting of $(r-s)$ $(2\times2)$-blocks of type 
$\left(\begin{array}{cc} 0&1\\-1&0\end{array}\right)$,
$D_3$ is a block-diagonal $\big((\frac{n}{2}-r)\times(\frac{n}{2}-r)\big)$-matrix consisting of $\frac{n}{4}-\frac{r}{2}$ 
$(2\times2)$-blocks of type $D_{3,k}=\left(\begin{array}{cc} \cos\theta_k&\sin\theta_k\\-\sin\theta_k&\cos\theta_k\end{array}\right)$, 
where $\cos\theta_k=\frac{\alpha_k}{|u|^2}$, $k=1,\ldots,\frac{n}{4}-\frac{r}{2}$, and $D_4=-D_3$. Notice that there is no matrix 
$D_1$ in \eqref{eq:tildeD} if $s=0$, i.~e. if $p$ is even. Analogously, there are no matrices $D_1$ and $D_2$ in \eqref{eq:tildeD} 
if $r=0$, and there are no matrices $D_3$ and $D_4$ if $k=0$.

Then 
\begin{equation}\label{eq:eD}
    e^{\tilde Dt}=\left(
      \begin{array}{cccc}
           e^{|u|D_1t}&0&0&0\\
           0&e^{|u|D_2t}&0&0\\
           0&0&e^{|u|^2D_3t}&0\\
           0&0&0&e^{|u|^2D_4t}
      \end{array}
\right).
\end{equation}
Notice that
\begin{gather*}
 D_1^l=\begin{cases}
        D_1,\; l - \mbox{odd},\\
        I_{2s}, \; l - \mbox{even},
       \end{cases}
 D_2^l=\begin{cases}
        D_2,\quad l=4d+1,\\
        -I_{2(r-s)}, \quad l=4d+2,\\
        -D_2,\quad l=4d+3,\\
        I_{2(r-s)},\quad l=4d,\quad d=0,1,2,\ldots,
       \end{cases}\\
 D_{3,k}^l=\left(\begin{array}{cc}  \cos l\theta_k&\sin l\theta_k\\
                                   -\sin l\theta_k&\cos l\theta_k
                \end{array}\right),\\
D_{4,k}^l=-D^l_{3.k},\quad k=1,\ldots,\frac{n}{4}-\frac{r}{2},
\end{gather*}
where $I_{2s}$ and $I_{2(r-s)}$ are identity matrices of dimensions $(2s\times 2s)$ and $2(r-s)\times 2(r-s)$ respectively.

Taking into account the formulae above we calculate each of the blocks in \eqref{eq:eD}
\begin{gather*}
    e^{|u|D_1t}=I_{\tiny 2s}+D_1|u|t+\frac{D_1^2|u|^2t^2}{2!}+\frac{D_1^3|u|^3t^3}{3!}+\ldots\\
   =\big(I_{2s}+\frac{(|u|t)^2I_{2s}}{2!}+\frac{(|u|t)^4I_{2s}}{4!}+\ldots\big)+
   \big(|u|tD_1+\frac{(|u|t)^3D_1}{3!}+\frac{(|u|t)^5D_1}{5!}+\ldots\big)\\
   =\cosh\big(|u|t\big)I_{2s}+\sinh\big(|u|t\big)D_1,\\
    e^{|u|D_2t}=I_{\tiny 2(r-s)}+D_2|u|t+\frac{D_2^2|u|^2t^2}{2!}+\frac{D_2^3|u|^3t^3}{3!}+\ldots\\
   =\big(I_{2(r-s)}-\frac{(|u|t)^2I_{2(r-s)}}{2!}+\frac{(|u|t)^4I_{2(r-s)}}{4!}-\ldots\big)+
   \big(|u|tD_2-\frac{(|u|t)^3D_2}{3!}+\frac{(|u|t)^5D_2}{5!}-\ldots\big)\\
   =\cos\big(|u|t\big)I_{2(r-s)}+\sin\big(|u|t\big)D_2,
\end{gather*}
\begin{gather*}
e^{|u|^2D_{3,k}t}=\left(\begin{array}{cc} \sum\limits_{l=0}^{\infty}\frac{\cos l\theta_k}{l!}(|u|^2t)^l & 
\sum\limits_{l=0}^{\infty}\frac{\sin l\theta_k}{l!}(|u|^2t)^l\\
-\sum\limits_{l=0}^{\infty}\frac{\sin l\theta_k}{l!}(|u|^2t)^l&
\sum\limits_{l=0}^{\infty}\frac{\cos l\theta_k}{l!}(|u|^2t)^l
\end{array}\right)\\
=\left(\begin{array}{cc}e^{|u|^2t\cos\theta_k}\cos(|u|^2t\sin\theta_k)& e^{|u|^2t\cos\theta_k}\sin(|u|^2t\sin\theta_k)\\
        -e^{|u|^2t\cos\theta_k}\sin(|u|^2t\sin\theta_k)&e^{|u|^2t\cos\theta_k}\cos(|u|^2t\sin\theta_k)
       \end{array}\right)\\
=e^{\alpha_k t}\left(\begin{array}{cc} \cos\beta_k t&\sin \beta_k t\\ -\sin\beta_k t&\cos\beta_k t\end{array}\right)
\end{gather*}
\begin{gather*}
e^{|u|^2D_{4,k}t}
=e^{-\alpha_k t}\left(\begin{array}{cc} \cos\beta_k t&\sin \beta_k t\\ -\sin\beta_k t&\cos\beta_k t\end{array}\right),\quad
k=1,\ldots,\frac{n}{4}-\frac{r}{2}.
\end{gather*}
Therefore, $e^{|u|^2D_3t}$ is a block-diagonal matrix consisting of $\frac{n}{4}-\frac{r}{2}$ $(2\times2)$-blocks of type 
$e^{|u|^2D_{3,k}t}$ and $e^{|u|D_4t}$ is a block-diagonal matrix consisting of $\frac{n}{4}-\frac{r}{2}$ $(2\times2)$-blocks of type 
$e^{|u|^2D_{4,k}t}$.

The solution of the system \begin{equation}\label{eq:sys_tilde}
                            \ddot{\tilde{\tilde v}}=\tilde D\dot{\tilde{\tilde v}},
                           \end{equation}
where $\tilde{\tilde v}=\tilde P\tilde v$, with the initial conditions $\tilde{\tilde v}(0)=0$, $\dot{\tilde{\tilde v}}(0)=
\dot{\tilde{\tilde v}}_0$, is  of the form 
$\tilde{\tilde v}(t)=\big(e^{\tilde Dt}-I_V\big)\tilde D^{-1}
\dot{\tilde{\tilde v}}_0$.
 Notice that $\tilde D^{-1}=\frac{1}
{|u|^2}\tilde D^t$, $\tilde D\tilde D^t=\tilde D^t\tilde D=|u|^2I_V$ and $D_1^t=D_1,D_2^t=-D_2$. 
Therefore, the solution of \eqref{eq:sys_tilde} reduces to 
\begin{equation}\label{eq:solut}
    \tilde{\tilde v}(t)=\frac{1}{|u|^2}\big(e^{\tilde Dt}-I_V\big)\tilde D^t\dot{\tilde{\tilde v}}_0.
\end{equation}

In order to write the solution more explicitely, let us make the following notations: 
let $\tilde{\tilde v}^1(t)$ denote the $2s$-dimensional vector consisting of $2s$ first 
coordinates from the vector $\tilde{\tilde v}(t)$, analogously let 
$\tilde{\tilde v}^2(t)$ denote the vector consisting of next $2(r-s)$ 
coordinates from the vector $\tilde{\tilde v}(t)$, let us denote also by $\tilde{\tilde v}^{3,k}(t)$ the $2$-dimensional part of vector 
$\tilde{\tilde v}(t)$ corresponding to each block $D_{3,k}$, and analogously, $\tilde{\tilde v}^{4,k}(t)$ -- part of 
$\tilde{\tilde v}(t)$ corresponding to block $D_{4,k}$. In analogous way, let $\dot{\tilde{\tilde v}}_0^1$ be the $2s$-dimensional vector of first $2s$ 
coordinates of the vector $\dot{\tilde{\tilde v}}_0$, $\dot{\tilde{\tilde v}}_0^2$ be the vector of next $2(r-s)$ 
coordinates of the vector $\dot{\tilde{\tilde v}}_0$ and 
let us denote by $\dot{\tilde{\tilde v}}^{3,k}_0$ the part of vector 
$\dot{\tilde{\tilde v}}_0$ corresponding to each block $D_{3,k}$, and analogously, $\dot{\tilde{\tilde v}}^{4,k}_0$ -- part of 
$\dot{\tilde{\tilde v}}_0$ corresponding to block $D_{4,k}$, $k=1,\ldots,\frac{n}{4}-\frac{r}{2}$.
 Then the solution \eqref{eq:solut} $\tilde{\tilde v}(t)=\big(\tilde{\tilde v}^1,\tilde{\tilde v}^2,
\tilde{\tilde v}^{3,1},\tilde{\tilde v}^{4,1},\ldots,\tilde{\tilde v}^{3,\frac{n}{4}-\frac{r}{2}},
\tilde{\tilde v}^{4,\frac{n}{4}-\frac{r}{2}}\big)(t)$ to the system \eqref{eq:sys_tilde} 
can be written in the form
\begin{equation}\label{eq:tildetildev}
 \begin{split}
    &\tilde{\tilde v}^1(t)=\frac{\dot{\tilde{\tilde v}}^1_0}{|u|^2}\big(\sinh(|u|t)I_{2s}+(\cosh(|u|t)-1)D_1\big),\\
    &\tilde{\tilde v}^2(t)=\frac{\dot{\tilde{\tilde v}}^2_0}{|u|^2}\big(\sin(|u|t)I_{2(r-s)}+(1-\cos(|u|t))D_2\big),\\
    &\tilde{\tilde v}^{3,k}(t)=\dot{\tilde{\tilde v}}^{3,k}_0e^{2\alpha_kt}\mathbb{B}_k,\\
    &\tilde{\tilde v}^{4,k}(t)=\dot{\tilde{\tilde v}}^{4,k}_0e^{-2\alpha_kt}\mathbb{B}_k,
\end{split}
\end{equation}
where matrix $\mathbb{B}_k=\frac{1}{|u|^2}\left(\begin{array}{cc}
    1-\cos(\beta_kt)&\sin(\beta_kt)\\-\sin(\beta_kt)&1-\cos(\beta_kt)
   \end{array}\right)$, $k=1,\ldots,\frac{n}{4}-\frac{r}{2}$.

For the horizontal part of the geodesic $\gamma(t)$ we make following observations: 
2-dimensional $\tilde{\tilde v}^1(t)$, corresponding to the cell in $\tilde D$ of type 
$D_1$, describes the geodesic which lives in a sub-Lorentzian 3-dimensional Heisenberg group, i.~e. Heisenberg group 
with the metric tensor 
$\left(\begin{array}{cc}-1&0\\0&1 \end{array}\right)$ on the horizontal bundle;
 2-dimensional projections of $\tilde{\tilde v}^2(t)$, corresponding to the cells in $\tilde D$ of type 
$D_2$, describe the geodesics which live in a sub-Riemannian 3-dimensional Heisenberg group with a Riemannian metric 
$\left(\begin{array}{cc}1&0\\0&1 \end{array}\right)$ on the horizontal bundle.
There are no analogues to $\tilde{\tilde v}^{3,k}(t)$ or $\tilde{\tilde v}^{3,k}(t)$ in Heisenberg group, but
together
$(\tilde{\tilde v}^{3,k}(t),\tilde{\tilde v}^{4,k}(t))$ for each $k$ they correspond to the geodesic which lives 
in a sub-semi-Riemannian $H$-type Quaternion group with the metric of index 2 on the horizontal subbundle and corresponds to the 
$(4\times4)$-dimensional block of type  $$\left(\begin{array}{cccc}\cos\theta_k&\sin\theta_k&0&0\\
-\sin\theta_k&\cos\theta_k&0&0\\0&0&-\cos\theta_k&-\sin\theta_k\\0&0&\sin\theta_k&-\cos\theta_k \end{array}\right),$$
(see \cite{KM2}).

Once $\tilde{\tilde v}(t)$ is calculated one can find the horizontal part $v(t)$ of the geodesic $\gamma(t)=(v(t),u(t))$ by 
\begin{equation}\label{eq:vvv}
v=P^{-1}\tilde P^{-1}\tilde{\tilde v}.
\end{equation}
To find a vertical part $u(t)$ we substitute $v=P^{-1}\tilde P^{-1}\tilde{\tilde v}$ into the condition \eqref{eq:uequation} and 
obtain that
$$\dot u+\frac12[P^{-1}\tilde P^{-1}\dot{\tilde{\tilde v}},\tilde{\tilde v}]=\dot u_0,$$
and, integrating, we get 
\begin{equation}\label{eq:u(t)} 
u(t)=-\frac12\int\limits_0^t[P^{-1}\tilde P^{-1}\dot{\tilde{\tilde v}},P^{-1}\tilde P^{-1}\tilde{\tilde v}]\,dt+\dot u_0t.
\end{equation}

\section{Example}\label{Example}

In this section we provide an illustrating example of how the theory above can be applied to concrete situation. 

Let us consider the $H$-type group with algebra $V\oplus U$, $V=\mathbb{R}^8$, $U=\mathbb{R}^7$,
and with nondegenerate metric $\eta$ of arbitrary index $p\leqslant4$ on $V$. 

Let $\gamma(t)=\big(v(t),u(t)\big)$ be a geodesic on such group. It is sufficient for us to find only the part $\tilde{\tilde v}(t)$ 
since $v(t)$ and $u(t)$ can be obtained from it by \eqref{eq:vvv} and \eqref{eq:u(t)}.

To escape extra indexes it would be convinient to denote initial vertical velocity $\dot u^0$  by $u=(u_1,\ldots,u_7)$.
Since $j(u)$ appears from the structure of the algebra and does not depend on the metric on $V$, we will base on the fact that 
matrix $j(u)$ for octonion $H$-type group in sub-Riemannian case is already calculated in \cite{CM}.
\begin{gather*}
    j(u)=\left(
      \begin{array}{cccccccc}
          0&u_1&u_2&u_3&u_4&u_5&u_6&u_7\\
          -u_1&0&u_3&-u_2&u_5&-u_4&-u_7&u_6\\
          -u_2&-u_3&0&u_1&u_6&u_7&-u_4&-u_5\\
          -u_3&u_2&-u_1&0&u_7&-u_6&u_5&-u_4\\
          -u_4&-u_5&-u_6&-u_7&0&u_1&u_2&u_3\\
          -u_5&u_4&-u_7&u_6&-u_1&0&-u_3&u_2\\
          -u_6&u_7&u_4&-u_5&-u_2&u_3&0&-u_1\\
          -u_7&-u_6&u_5&u_4&-u_3&-u_2&u_1&0
      \end{array}
    \right),
\end{gather*} where $u=(u_1,\ldots,u_7)\in\mathbb{R}^7$.
Matrices $j_{\alpha}$ are taken from \cite{CM} and have the following forms:
\begin{gather*}
 j_1=\left(
     \begin{array}{cccccccc}
        0&1&0&0&0&0&0&0\\
        \!\!\!\!{}-\!\!{}1&0&0&0&0&0&0&0\\
        0&0&0&1&0&0&0&0\\
        0&0&\!\!\!\!{}-\!\!{}1&0&0&0&0&0\\
        0&0&0&0&0&1&0&0\\
        0&0&0&0&\!\!\!\!{}-\!\!{}1&0&0&0\\
        0&0&0&0&0&0&0&\!\!\!\!{}-\!\!{}1\\
        0&0&0&0&0&0&1&0
     \end{array}
\right),\qquad
 j_2=\left(
     \begin{array}{cccccccc}
        0&0&1&0&0&0&0&0\\
        0&0&0&\!\!\!\!{}-\!\!{}1&0&0&0&0\\
        \!\!\!\!{}-\!\!{}1&0&0&0&0&0&0&0\\
        0&1&0&0&0&0&0&0\\
        0&0&0&0&0&0&1&0\\
        0&0&0&0&0&0&0&1\\
        0&0&0&0&\!\!\!\!{}-\!\!{}1&0&0&0\\
        0&0&0&0&0&\!\!\!\!{}-\!\!{}1&0&0
     \end{array}
\right),
\end{gather*}
\begin{gather*}
  j_3=\left(
     \begin{array}{cccccccc}
        0&0&0&1&0&0&0&0\\
        0&0&1&0&0&0&0&0\\
        0&\!\!\!\!{}-\!\!{}1&0&0&0&0&0&0\\
        \!\!\!\!{}-\!\!{}1&0&0&0&0&0&0&0\\
        0&0&0&0&0&0&0&1\\
        0&0&0&0&0&0&\!\!\!\!{}-\!\!{}1&0\\
        0&0&0&0&0&0&1&0\\
        0&0&0&0&\!\!\!\!{}-\!\!{}1&0&0&0
     \end{array}
\right),\qquad
j_4=\left(
     \begin{array}{cccccccc}
        0&0&0&0&1&0&0&0\\
        0&0&0&0&0&\!\!\!\!{}-\!\!{}1&0&0\\
        0&0&0&0&0&0&\!\!\!\!{}-\!\!{}1&0\\
        0&0&0&0&0&0&0&\!\!\!\!{}-\!\!{}1\\
        \!\!\!\!{}-\!\!{}1&0&0&0&0&0&0&0\\
        0&1&0&0&0&0&0&0\\
        0&0&1&0&0&0&0&0\\
        0&0&0&1&0&0&0&0
     \end{array}
\right),
\end{gather*}
\begin{gather*}
 j_5=\left(
     \begin{array}{cccccccc}
        0&0&0&0&0&1&0&0\\
        0&0&0&0&1&0&0&0\\
        0&0&0&0&0&0&0&\!\!\!\!{}-\!\!{}1\\
        0&0&0&0&0&0&1&0\\
        0&\!\!\!\!{}-\!\!{}1&0&0&0&0&0&0\\
        \!\!\!\!{}-\!\!{}1&0&0&0&0&0&0&0\\
        0&0&0&\!\!\!\!{}-\!\!{}1&0&0&0&0\\
        0&0&1&0&0&0&0&0
     \end{array}
\right),\qquad
j_6=\left(
     \begin{array}{cccccccc}
        0&0&0&0&0&0&1&0\\
        0&0&0&0&0&0&0&1\\
        0&0&0&0&1&0&0&0\\
        0&0&0&0&0&\!\!\!\!{}-\!\!{}1&0&0\\
        0&0&\!\!\!\!{}-\!\!{}1&0&0&0&0&0\\
        0&0&0&1&0&0&0&0\\
        \!\!\!\!{}-\!\!{}1&0&0&0&0&0&0&0\\
        0&\!\!\!\!{}-\!\!{}1&0&0&0&0&0&0
     \end{array}
\right),
\end{gather*}
\begin{gather*}
 j_7=\left(
     \begin{array}{cccccccc}
        0&0&0&0&0&0&0&1\\
        0&0&0&0&0&0&\!\!\!\!{}-\!\!{}1&0\\
        0&0&0&0&0&1&0&0\\
        0&0&0&0&1&0&0&0\\
        0&0&0&\!\!\!\!{}-\!\!{}1&0&0&0&0\\
        0&0&\!\!\!\!{}-\!\!{}1&0&0&0&0&0\\
        0&1&0&0&0&0&0&0\\
        \!\!\!\!{}-\!\!{}1&0&0&0&0&0&0&0
     \end{array}
\right)
\end{gather*}

Let $j_1,\ldots,j_7$ be endomorphisms from subsection \ref{matrices} satisfying \eqref{eq:proper} and such that 
$$j(u)=u_1j_1+\ldots+u_7j_7.$$ 
Denote as before the matrix $\eta j(u)$ by symbol $A$. Denote also 
$|u|^2=u_1^2+u_2^2+u_3^3+u_4^2+u_5^2+u_6^2+u_7^2$.
 
Let us find the characteristic polynomials and eigenvalues of $A$ for all $p=1,\ldots,4$.

If $p=1$, then
\begin{gather*}
    p^1_A(\lambda)=\lambda^8+2|u|^2\lambda^6-2|u|^6\lambda^2-|u|^8,
\end{gather*}
and eigenvalues of $A$ are $\pm |u|,\,\pm i|u|,\,\pm i|u|,\,\pm i|u|$. 
Notice that all matrices $\eta j_1,\ldots,\eta j_7$ have the same eigenvalues $\pm1,\pm i,\pm i,\pm i$.

We see that if $p=1$, then $s=1$, $r=4$ and $k=0$ in \eqref{eq:real}--\eqref{eq:complex} and, therefore, 
$\tilde{\tilde v}(t)=\big(\tilde{\tilde v}^1(t),\tilde{\tilde v}^2(t)\big)$, $
\dot{\tilde{\tilde v}}(0)=\big(\dot{\tilde{\tilde v}}^1_0,\dot{\tilde{\tilde v}}^2_0\big)$.
By \eqref{eq:tildetildev} we get
\begin{gather*}
    \tilde{\tilde v}^1(t)=\frac{\dot{\tilde{\tilde v}}^1_0}{|u|^2}\big(\sinh|u|tI_2+(\cosh|u|t-1)D_1\big),\\
    \tilde{\tilde v}^2(t)=\frac{\dot{\tilde{\tilde v}}^2_0}{|u|^2}\big(\sin|u|tI_6+(1-\cos|u|t)D_2\big),\\
\end{gather*}
where $D_1,D_2$ are $(6\times6)$-matrix described in \eqref{eq:tildeD}, $I_2,I_6$ -- identity matrices.

If $p=2$, then
\begin{gather*}
    p^2_A(\lambda)=\lambda^8+4u_1^2\lambda^6+2|u|^2\big(a^2+2(u_1^2-u_2^2-\ldots-u_7^2)\big)\lambda^4+4|u|^4u_1^2\lambda^2+|u|^8,
\end{gather*}
and eigenvalues are $\pm(iu_1\pm\sqrt{u_2^2+\ldots+u_7^2}),\,\pm i|u|,\,\pm i|u|$. 
Notice that matrix $\eta j_1$ has eigenvalues $\pm i,\pm i, \pm i,\pm i$ and matrices
$\eta j_2,\ldots \eta j_7$ have eigenvalues $\pm 1,\pm 1,\pm i,\pm i$.

We see that if $p=2$, then $s=0$, $r=2$ and $k=1$, $\alpha_1=\sqrt{u_2^2+\ldots+u_7^2}$, $\beta_1=u_1$ in \eqref{eq:real}--\eqref{eq:complex} and, therefore, 
$\tilde{\tilde v}(t)=\big(\tilde{\tilde v}^2(t),\tilde{\tilde v}^{3,1}(t),\tilde{\tilde v}^{4,1}(t)\big)$, $
\dot{\tilde{\tilde v}}(0)=\big(\dot{\tilde{\tilde v}}^2_0,\dot{\tilde{\tilde v}}^{3,1}_0,\dot{\tilde{\tilde v}}^{4,1}_0\big)$.
By \eqref{eq:tildetildev} we get
\begin{gather*}
    \tilde{\tilde v}^2(t)=\frac{\dot{\tilde{\tilde v}}^1_0}{|u|^2}\big(\sin|u|tI_4+(1-\cos|u|t)D_2\big),\\
    \tilde{\tilde v}^{3,1}(t)=\dot{\tilde{\tilde v}}^{3,1}_0e^{2\alpha_1t}\mathbb{B}_1,\\
    \tilde{\tilde v}^{4,1}(t)=\dot{\tilde{\tilde v}}^{4,1}_0e^{-2\alpha_1t}\mathbb{B}_1,\\
\end{gather*}
where $D_2$ are $(4\times4)$-matrix described in \eqref{eq:tildeD}, $I_4$ is $(4\times4)$-identity matrix and 
\begin{equation}\label{eq:fatB}
  \mathbb{B}_1=\frac{1}{|u|^2}\left(\begin{array}{cc} 1-\cos\beta_1t&\sin\beta_1t\\-\sin\beta_1t&1-\cos\beta_1t\end{array}
\right).                                                                                                            
                                                                                                              \end{equation}

If $p=3$, then
\begin{gather*}
    p^3_A(\lambda)=\lambda^8+2(u_1^2+u_2^2+u_3^2-u_4^2-\ldots-u_7^2)\lambda^6-2|u|^4(u_1^2+u_2^2+u_3^2-u_4^2
   -\ldots-u_7^2)\lambda^2-|u|^8,
\end{gather*}
and eigenvalues are $\pm |u|,\,\pm i|u|,\,\pm(i\sqrt{u_1^2+u_2^2+u_3^2}\pm\sqrt{u_4^2+\ldots+u_7^2})$. 
Notice that $\eta j_1,\eta j_2,\eta j_3$ have eigenvalues $\pm 1,\pm i,\pm i,\pm i$, and matrices $\eta j_4,\ldots,\eta j_7$ 
have eigenvalues $\pm i,\pm 1,\pm 1,\pm 1$.

We see that if $p=3$, then $s=1$, $r=2$ and $k=1$, $\alpha_1=\sqrt{u_4^2+\ldots+u_7^2}$, $\beta_1=\sqrt{u_1^2+u_2^2+u_3^2}$, and 
$\tilde{\tilde v}(t)=\big(\tilde{\tilde v}^2(t),\tilde{\tilde v}^2(t),\tilde{\tilde v}^{3,1}(t),\tilde{\tilde v}^{4,1}(t)\big)$,
$\dot{\tilde{\tilde v}}(0)=\big(\dot{\tilde{\tilde v}}^1_0,\dot{\tilde{\tilde v}}^2_0,
\dot{\tilde{\tilde v}}^{3,1}_0,\dot{\tilde{\tilde v}}^{4,1}_0\big)$. Then, by \eqref{eq:tildetildev}
\begin{gather*}
    \tilde{\tilde v}^1(t)=\frac{\dot{\tilde{\tilde v}}^1_0}{|u|^2}\big(\sinh|u|tI_2+(\cosh|u|t-1)D_1\big),\\
    \tilde{\tilde v}^2(t)=\frac{\dot{\tilde{\tilde v}}^2_0}{|u|^2}\big(\sin|u|tI_2+(1-\cos|u|t)D_2\big),\\
    \tilde{\tilde v}^{3,1}(t)=\dot{\tilde{\tilde v}}^{3,1}_0e^{2\alpha_1t}\mathbb{B}_1,\\
    \tilde{\tilde v}^{4,1}(t)=\dot{\tilde{\tilde v}}^{4,1}_0e^{-2\alpha_1t}\mathbb{B}_1,\\
\end{gather*}
where $\mathbb{B}_1$ has the same form as in \eqref{eq:fatB} with 
$\beta_1=\sqrt{u_1^2+u_2^2+u_3^2}$, $D_1$ and $D_2$ are $(2\times2)$-matrices 
represented in \eqref{eq:tildeD}.

If $p=4$, then
\begin{gather*}
    p^4_A(\lambda)=\lambda^8+4(u_1^2+u_2^2+u_3^2-u_4^2-\ldots-u_7^2)\lambda^6\\+2\big(
    |u|^4+2(u_1^2+u_2^2+u_3^2-u_4^2-\ldots-u_7^2)^2\big)\lambda^4+4|u|^4(u_1^2+u_2^2+u_3^2-u_4^2-\ldots-u_7^2)\lambda^2+|u|^8,
\end{gather*}
and eigenvalues are $\pm(i\sqrt{u_1^2+u_2^2+u_3^2}\pm\sqrt{u_4^2+\ldots+u_7^2})$, 
$\pm(i\sqrt{u_1^2+u_2^2+u_3^2}\pm\sqrt{u_4^2+\ldots+u_7^2})$.
Notice that matrices $\eta j_1,\eta j_2,\eta j_3$ have only purely imaginary eigenvalues $\pm i,\pm i,\pm i,\pm i$, 
and matrices $\eta j_4,\ldots,\eta j_7$ 
have only real eigenvalues $\pm 1,\pm 1,\pm 1,\pm 1$.

We see that is $p=4$, then $s=0$, $r=0$, $k=2$ and $\alpha_1=\alpha_2=\sqrt{u_4^2+\ldots+u_7^2}$, $\beta_1=\beta_2=\sqrt{u_1^2+u_2^2+u_3^2}$. Then 
$\tilde{\tilde v}(t)=\big(\tilde{\tilde v}^{3,1}(t),\tilde{\tilde v}^{4,1}(t),
\tilde{\tilde v}^{3,2}(t),\tilde{\tilde v}^{4,2}(t)\big)$ and initial velocity is 
$\dot{\tilde{\tilde v}}(0)=\big(\dot{\tilde{\tilde v}}^{3,1}_0,\dot{\tilde{\tilde v}}^{4,1}_0,
\dot{\tilde{\tilde v}}^{3,2}_0,\dot{\tilde{\tilde v}}^{4,2}_0\big)$. By formulae \eqref{eq:tildetildev} we have 
\begin{gather*}
     \tilde{\tilde v}^{3,1}(t)=\dot{\tilde{\tilde v}}^{3,1}_0e^{2\alpha_1t}\mathbb{B}_1,\\
     \tilde{\tilde v}^{4,1}(t)=\dot{\tilde{\tilde v}}^{4,1}_0e^{-2\alpha_1t}\mathbb{B}_1,\\
     \tilde{\tilde v}^{3,2}(t)=\dot{\tilde{\tilde v}}^{3,2}_0e^{2\alpha_1t}\mathbb{B}_1,\\
    \tilde{\tilde v}^{4,2}(t)=\dot{\tilde{\tilde v}}^{4,2}_0e^{-2\alpha_1t}\mathbb{B}_1,\\
\end{gather*}
In fact, the author's guess is that for all sub-semi-Riemannian $H$-type groups satisfying $j^2$-condition all complex eigenvalues 
$\alpha_k+i\beta_k$ are the same $\forall k$ for every fixed index $p$.
Notice that in case $p$ is even there are no real eigenvalues, i.~e. $s=0$ in \eqref{eq:real}, and in case $p$ is odd $\pm |u|$ are the only couple of 
real eigenvalues, i.~e. $s=1$.\\

\section{Possible open problems}

As a continuation of work in this direction it would be interesting to consider the generalization of matrices $A$ to matrices of 
order $n=\rho^{a}k$, where $\rho$ is a prime number and $k$ is prime to $\rho$, and to consider the corresponding analogue of $H$-type groups. 
In this case each matrix $j_{\alpha}$  would represent a $p$-th root of $-1$ (see \cite{Little}). The number of such matrices 
in a set of $n$-rowed matrices that satisfy the conditions
$$j^p_{\alpha}=-I,\qquad j_{\alpha}j_{\beta}=-j_{\beta}j_{\alpha}, \quad(\alpha\neq\beta),$$ 
is not more than $2a+1$.

Another interesting research line could be considering vector spaces $U$ and $V$ both with nondegenerate metrics $\langle
\cdot,\cdot\rangle_U$ and $\langle\cdot,\cdot\rangle_V$. The description of corresponding $H$-type groups would be probably simpler 
than in the present article, since the construction of representations of Clifford algebras $\Cl(U)$ are less tricky in the case of 
nondegenerate indefinite product on $U$.
\\

{\bf Acknowledgments}.\\

I wish to thank my advisor Irina Markina and my colleague Mauricio Godoy Molina for useful suggestions and fruitful discussions
 during the work.

\end{document}